\documentclass[a4paper,11pt]{article}
\usepackage[utf8]{inputenc}
\usepackage[T1]{fontenc}

\usepackage[affil-it]{authblk}
\usepackage{indentfirst}
\usepackage{enumitem}
\usepackage{amsmath,amsthm,amssymb}
\usepackage{aliascnt}
\usepackage{hyperref}
\usepackage[capitalise]{cleveref}
\usepackage{ragged2e}
\usepackage{adjustbox}
\usepackage[autostyle, english = american]{csquotes}
\MakeOuterQuote{"}

\usepackage{tikz-cd}
\usetikzlibrary{decorations.markings,intersections}
\usetikzlibrary{matrix}
\usetikzlibrary{arrows}
\usetikzlibrary{arrows.meta}
\usetikzlibrary{decorations.pathmorphing,shapes}
\usetikzlibrary{decorations.markings}
\tikzset{shorten <>/.style={shorten >=#1,shorten <=#1}}
\tikzset{every picture/.prefix code=\DisableQuotes}

\tikzset{spanmap/.style={
            decoration={markings,
            mark= at position 0.5 with {
                  \node[transform shape] (tempnode) {$|$};
                  }
              },
              postaction={decorate}
}
}

\usepackage{quiver}
\usepackage{xcolor}

\usepackage[mathcal]{euscript}

\usepackage{ stmaryrd }
\usepackage{ dsfont }
\usepackage{upgreek}
\usepackage{ mathrsfs }
\usepackage[bb=dsserif]{mathalpha}
\usepackage{relsize}
\usepackage[backend=biber,style=alphabetic]{biblatex}

\let\oldtheorem\newtheorem
\RenewDocumentCommand{\newtheorem}{s m o m O{}}{%
\IfBooleanTF{#1}%
{\oldtheorem{#2}{#4}}%
{\IfNoValueTF{#3}{\oldtheorem{#2}{#4}[#5]}%
{\newaliascnt{#2}{#3}%
\oldtheorem{#2}[#2]{#4}%
\aliascntresetthe{#2}}}}

\newtheorem{theorem}{Theorem}[section]
\newtheorem{proposition}[theorem]{Proposition}
\newtheorem{lemma}[theorem]{Lemma}

\theoremstyle{definition}
\newtheorem{definition}{Definition}[section]

\theoremstyle{remark}
\newtheorem{remark}{Remark}[section]
\newtheorem{example}[remark]{Example}

\newcommand{\catname}[1]{\mathbf{#1}}

\begin{document}

\title{\textbf{Interpreting type theory in a quasicategory: a Yoneda approach}}
\author{El Mehdi Cherradi}
\affil{IRIF - CNRS - Universit\'e Paris Cit\'e \\MINES ParisTech - Universit\'e PSL}
\date{}

\maketitle

\begin{abstract}
We make use of a higher version of the Yoneda embedding to construct, from a given quasicategory, a tribe, as a subcategory of a well-behaved simplicial model category, that presents the same $(\infty,1)$-category as the former quasicategory. 
We then show that, when the quasicategory is locally cartesian closed, it is possible to further endow such a tribe with enough structure for it to provide a model of Martin-Löf type theory with $\Pi$-types.
This mapping procedure restricts so that elementary higher topoi yield models of homotopy type theory.
\end{abstract}

\tableofcontents

\newpage 

\section*{Introduction}
\addcontentsline{toc}{section}{Introduction}

The connection between logics and categories is important and has been studied, notably using a fibrational approach (see \cite{jacobs1999}), which has resulted in categorical notions such as that of comprehension categories to provide a categorical account for the syntax of (dependent) type theory. From a semantical point-of-view, the correspondence ranges from boolean algebras as models for propositional logic, to locally cartesian closed categories for (extensional)  Martin-Löf type theory (\cite{hofmann1994}).

At the same time, the study of categories has been shaped by the ideas of weakening some axioms given as equations by shifting from an equality to an equivalence (for instance as for bicategories), which has led to make use of the general theory of abstract homotopy in order to develop homotopy coherent versions of many categorical constructions. $(\infty,1)$-categories (\cite{joyal2008}) consist, for instance, in a homotopy-enabled version of usual categories (referred to as $1$-categories). The similar trend that has affected type theory led to homotopy type theory, where a proof of equality between terms is thought as a path in a space, hence providing types with the structure of spaces (or $\infty$-groupoids) from the point of view of homotopy theory.

Motivated by the $1$-categorical result on models of extensional Martin-Löf type theory mentioned previously, the natural question that arose about the link between $(\infty,1)$-categories and homotopy type theory has recently found some answers extending this logical correspondence to quasicategories, through the simplicial model for homotopy type theory (\cite{kl2012}) and more generally the internal language of $(\infty,1)$-topoi, the higher version of  Grothendieck topoi (\cite{shulman2019}).

Among the different models of $(\infty,1)$-categories, quasicategories have the specific property to come with an inbuilt notion of homotopy, making the usual categorical constructions (for instance limits) compatible with homotopy. %
However, while this makes quasicategories a very convenient framework to talk about homotopy coherent properties, there is a mismatch with the usual way to the interpret the syntax of type theory, using strict categorical structures (as reviewed in \cite{curien2014}), even in the case of homotopy type theory.

This mismatch can be thought of under the angle of coherence. Indeed, morphisms (which are central to a categorical interpretation of type theory) in a quasicategory exist, but there is no internal way to reason about them strictly, as everything works up to homotopy: for instance, composition of morphisms is "only" defined up to homotopy. Therefore, if we want to deal with morphisms and their composites, we would need to compose them in a coherent fashion, so that associativity and unitality hold strictly. This looks difficult, if not impossible, to achieve directly in the framework of quasicategories. This coherence problem is reminiscent of the splitting procedure for a Grothendieck fibration introduced by \cite{benabou1980}, see also \cite{giraud}, whose importance for solving the usual coherence problem arising from pullbacks so as to model substitution in type theory has been noticed, see \cite{hofmann1994}. %

Unsurprisingly, interpreting type theory in a higher topos involves a rigidification step to make sense of the expected strict rules of type theory. This can account for the fact that the current results only apply to $(\infty,1)$-categories which can be presented by a model category: such a model category works as rigidified substitute of, let say, a quasicategory, and presents the very same data in a $1$-categorical way (which is more directly suitable to interpret type theory).

Following Bénabou's idea, we may want to rely on a categorical device with inbuilt composition to replace morphisms (and objects). But this is precisely the role played by the slice categories, or, equivalently, the representable presheaves! Indeed, the Yoneda embedding allows us to see any small category $\mathbf{C}$ as a full subcategory of $\catname{Set}^{\mathbf{C}^{op}}$ (or of $\catname{Cat}_{/\mathbf{C}}$ through the Grothendieck construction).

This paper puts this idea at work in the setting of higher categories. This gives rise to a general construction to rigidify a quasicategory (that is to turn it into an equivalent, as an $(\infty,1)$-category, simplicial category). This rigidification can be exploited to interpret in full generality a large fragment MLTT in any quasicategory. This gives access to a type-theoretic account of the quasicategory and its logical properties. Although the latter is the true aim of the paper, the rigidification process is interesting on its own as it gives a more concrete alternative to the rigidification functor $\mathfrak{C}$ defined by Lurie (\cite{lurie2009}) as the left adjoint to the simplicial nerve $\mathbf{N}_\Delta$ (see also \cite{dugger2011}). 

Our main result consists in providing models of Martin-Löf type theory for a more general class of quasicategories (the so-called internal logic of a quasicategory) than in the current literature. Having shown that structure such as limits, dependent products and object classifiers in a quasicategory can be carried over to a corresponding simplicial category to model the associated type-theoretic notions, we most notably observe that elementary higher topoi (that is, locally cartesian closed quasicategories with finite limits and colimits, a subobject classifier and enough universes, see \cite{rasekh2018eth}) provide models for homotopy type theory. 

\section*{Overview of the paper}
\addcontentsline{toc}{section}{Overview of the paper}

In the first section, we construct the rigidification of a given quasicategory, and observe that it is indeed equivalent, as an $(\infty,1)$-category, to this quasicategory (precisely the simplicial nerve of the rigidification is related to the quasicategory which we start with by a zig-zag of equivalences of quasicategories).

In the second section, we study the case where the quasicategory we start from is locally cartesian closed: we show that we can add to the rigidification some strict structure (mainly pullbacks and dependent products along fibrations as well as path objects) while preserving the equivalence of $(\infty,1)$-categories. At this point, the rigidification forms a $\pi$-tribe, and we can see in Theorem \ref{t1} that the internal logic of a locally cartesian closed category is a type theory with dependent sums and products, and intensional identity types.

The third section investigates the case of univalent universes, showing how they can be carried over to the rigidification of the quasicategory.

The fourth section briefly discusses a strategy to use initial structures in the quasicategory to construct (higher) inductive types.

In the fifth section, we are able to conclude in Theorem \ref{main} that elementary higher topoi further model homotopy type theory (with univalent universes and many inductive types).

\section{A Yoneda-style rigidification}

Quasicategories and (fibrant) simplicial categories are known to provide equivalent models for $(\infty,1)$-categories. This is witnessed by a Quillen equivalence, where the right Quillen functor is the usual simplicial nerve (also called homotopy coherent nerve) of a simplicial category that we write $$\mathbf{N}_\Delta : \catname{Cat}_\Delta \to \catname{SSet}$$
\noindent
that goes from simplicially enriched categories (with the Bergner model structure) to simplicial sets (with the Joyal model structure).
We will write $\mathfrak{C}$ for its left adjoint, following Lurie's notation in \cite{lurie2009}.

From now on, in this chapter, we fix a quasicategory $\mathcal{C}$.
The unit $$\mathcal{C} \to \mathbf{N}_\Delta(\mathfrak{C}(\mathcal{C}))$$ is a categorical equivalence, that is a weak equivalence in the Joyal model structure for simplicial sets, or, more concretely, a functor that is fully faithful  (i.e, induces weak equivalences of simplicial sets on the hom-spaces) and that is essentially surjective on objects (every object in the codomain is equivalent to an object in the image).

We think of $\mathfrak{C}(\mathcal{C})$ as a rigidified version of $\mathcal{C}$ following Dugger's terminology, see \cite{dugger2011}.
Lurie showed that, in particular, there is a canonical equivalence of simplicial sets (with respect to the Quillen model structure) induced by the previous unit:

$$Hom_{\mathcal{C}}(x,y) \to Hom_{\mathbf{N}_\Delta(\mathfrak{C}(\mathcal{C}))}(x,y)$$

\noindent
However, the precise description of the hom-spaces $Hom_{\mathfrak{C}(\mathcal{C})}(x,y)$ is not very easy to work with, although a nice description relying on the notion of \textit{necklace} is spelled out in \cite{dugger2011}. 

As mentioned in the introduction, an alternative way is suggested by the following result, a simple consequence of the fibrational Yoneda lemma, that is, the Yoneda lemma as seen through the Grothendieck construction, which states that, for $\mathbf{C}$ a category and $F$ a discrete fibration over $\mathbf{C}$, the canonical map

$$Hom_{\catname{Cat}_{/\mathbf{C}}}(\mathbf{C}_{/x},F) \simeq F^{-1}(x)$$

is a natural bijection (see also Theorem 3.1 in \cite{streicher2018fibered}).

\begin{proposition}
\label{1ff}
 Let $\mathbf{C}$ be a small category. Then the functor $\mathbf{C} \to \catname{Cat}_{/\mathbf{C}}$ defined on objects by $x \mapsto \mathbf{C}_{/x}$ is fully faithful and thus exhibits $\mathbf{C}$ as a full subcategory of $\catname{Cat}_{/\mathbf{C}}$.
\end{proposition}

We therefore propose, as a first step, the following construction:

\begin{definition}
  We define $\overline{\mathcal{C}}_{\mathbf{fb}}$ as the full simplicial subcategory of $\textbf{SSet}_{/\mathcal{C}}$ with objects the slice categories $\mathcal{C}_{/x} \to \mathcal{C}$ equipped with the projection, for every object $x$ of $\mathcal{C}$. Note that this is a subcategory of fibrant-cofibrant objects with respect to the contravariant model structure on $\textbf{SSet}_{/\mathcal{C}}$.
\end{definition}

The construction of $\overline{\mathcal{C}}_{\mathbf{fb}}$ provides a rather concrete rigidified version of $\mathcal{C}$, however the contravariant model structure (which corresponds in a sense to the projective global model structure on simplicial presheaves) lacks some important properties that we will need in the rest of the chapter.

For this purpose, our preferred alternative will be the following variation. This allows us to trade the "concreteness" of $\overline{\mathcal{C}}_{\mathbf{fb}}$ for a better homotopy behavior (given by the global injective model structure on simplicial presheaves, which is an \textit{excellent} model structure in the sense of \cite{ls2020}).

\begin{definition}
 For every object $x$ of $\mathcal{C}$, which we also think of as an object of $\mathfrak{C}(\mathcal{C})$, we write $\overline{y}(x)$ for an injectively fibrant replacement (hence a fibrant-cofibrant object) in $\mathbf{SSet}^{\mathfrak{C}(\mathcal{C}^{op})}$ of the representable simplicial functor $Hom(-,x)$.
 We write $\overline{\mathcal{C}}$ for the full subcategory of $\mathbf{SSet}^{\mathfrak{C}(\mathcal{C}^{op})}$ spanned by the $\overline{y}(x)$.
 
\end{definition}

We shall now prove the following expected result:%

\begin{proposition}
\label{rigidification}
 The simplicial nerve $\mathbf{N}_\Delta(\overline{\mathcal{C}}_{\mathbf{fb}})$ is equivalent to $\mathcal{C}$.
 Similarly, the simplicial nerve of $\overline{\mathcal{C}}$ is equivalent to $\mathcal{C}$.
\end{proposition}

The assignment $x \mapsto \mathcal{C}_{/x}$ does not seem to extend to a functor of quasicategories (i.e., a morphism of simplicial sets) in a canonical way.

However, the proof of Proposition \ref{1ff} suggests that we should rely on both the $(\infty,1)$-Yoneda lemma, as well as the higher version of the Grothendieck construction (the (un)straightening construction). Both of these constructions are studied in detail in \cite{lurie2009}.

Actually, we will deduce the expected result stated in Proposition \ref{rigidification} directly from the following one, which is due to Lurie (Proposition 5.1.1.1 in \cite{lurie2009}). 
We write $\mathcal{S}$ for the quasicategory of spaces (obtained by applying the simplicial nerve to the simplicial category of Kan complexes), $\mathbf{RFib}(\mathcal{C})$ for the full simplicial subcategory of $\mathbf{SSet}_{/\mathcal{C}}$ spanned by the right fibrations (which are exactly the fibrant-cofibrant objects for the contravariant model structure), and $(\mathbf{SSet}^{\mathfrak{C}(\mathcal{C}^{op})})^\circ$ for the full subcategory of $\mathbf{SSet}^{\mathfrak{C}(\mathcal{C}^{op})}$ (the contravariant simplicial functors from $\mathfrak{C}(\mathcal{C})$ to $\mathbf{SSet}$) spanned by the objects that are fibrant-cofibrant for the global projective model structure. The statement is then as follows:

\begin{proposition}
 There is a span of equivalences 
 
 $$Fun(\mathcal{C}^{op}, \mathcal{S}) \leftarrow \mathbf{N}_\Delta((\mathbf{SSet}^{\mathfrak{C}(\mathcal{C}^{op})})^\circ) \rightarrow \mathbf{N}_\Delta(\mathbf{RFib}(\mathcal{C}))$$
\end{proposition}

We will see that it is enough to study the restriction of these two maps to the full subcategory of $\mathbf{N}_\Delta((\mathbf{SSet}^{\mathfrak{C}(\mathcal{C}^{op})})^\circ)$ spanned by (fibrant-cofibrant replacements of) the representable simplicial functors $Hom_{\mathfrak{C}(\mathcal{C})}(-,x)$ (for $x$ an object of $\mathcal{C}$). Note that, here, we are abusing notations not to make distinctions between objects of a simplicial category and objects of its simplicial nerve, and, similarly, objects of a simplicial set $\mathbf{S}$ and objects of $\mathfrak{C}(\mathcal{S})$.

First, we recall the statement of two important results, which can also be found in \cite{lurie2009}:

\begin{theorem}
 There is a Quillen adjunction 

\[\begin{tikzcd}
	{\mathbf{SSet}_{/\mathcal{C}}} && {\mathbf{SSet}^{\mathfrak{C}(\mathcal{C}^{op})}}
	\arrow["{\mathbf{St}}", shift left=2, from=1-1, to=1-3]
	\arrow["{\mathbf{Un}}", shift left=2, from=1-3, to=1-1]
\end{tikzcd}\]

 \noindent
 between $\mathbf{SSet}_{/\mathcal{C}}$ equipped with the contravariant model structure  and $\mathbf{SSet}^{\mathfrak{C}(\mathcal{C}^{op})}$ equipped with the global projective model structure.
 The right (resp. left) adjoint is called the unstraightening (resp. straightening).
\end{theorem}

\begin{definition}
 Consider a fibrant replacement functor $R : \mathbf{SSet} \to \mathbf{SSet}$ for the Quillen model structure. The assignment $$(X,Y) \to R (Hom_{\mathfrak{C}(\mathcal{C})}(X,Y))$$
 defines a simplicial functor from $\mathfrak{C}(\mathcal{C}^{op}) \times \mathfrak{C}(\mathcal{C})$ to the simplicial category $\mathbf{Kan}$ of Kan complexes. Composing on the left with the natural map $$\mathfrak{C}(\mathcal{C}^{op} \times \mathcal{C}) \to \mathfrak{C}(\mathcal{C}^{op}) \times \mathfrak{C}(\mathcal{C})$$ and transposing along the adjunction $\mathfrak{C} \dashv \mathbf{N}_\Delta$, we get a map of simplicial sets: $$\mathcal{C}^{op} \times \mathcal{C} \to \mathcal{S}$$
 
 We define the Yoneda embedding to be the corresponding map $$j : \mathcal{C} \to Fun(\mathcal{C}^{op}, \mathcal{S})$$
\end{definition}

\begin{proposition}
The Yoneda embedding is fully faithful and preserves (small) limits.
\end{proposition}

\begin{proof}[Proof of Proposition \ref{rigidification}]

We will put together the results, due to Lurie, that were mentioned before, to construct canonical zig-zags of equivalence between the quasicategories we consider.

Firstly, the morphism $$\mathbf{N}_\Delta(\mathbf{SSet}^{\mathfrak{C}(\mathcal{C}^{op})})^\circ)  \rightarrow \mathbf{N}_\Delta(\mathbf{RFib}(\mathcal{C}))$$ is obtained by applying the nerve to a functor of simplicial categories $$g: (\mathbf{SSet}^{\mathfrak{C}(\mathcal{C}^{op})})^\circ \rightarrow \mathbf{RFib}(\mathcal{C})$$ induced by the unstraightening functor (see \cite{lurie2009}, or \cite{hhr2021}, for details on the latter functor).

The right fibration $\mathbf{Un}(F')$ associated by this unstraightening functor with (a fibrant-cofibrant replacement $F'$ of) a representable functor $$F = Hom_{\mathfrak{C}(\mathcal{C})}(-,x)$$ is fiberwise equivalent to the canonical right fibration $\mathcal{C}_{/x} \to \mathcal{C}$. This is implied by the fact that there is a canonical weak homotopy equivalence

$$\mathbf{St}(\mathcal{C}_{/x})(y) \to Hom_{\mathfrak{C}(\mathcal{C})}(y,x)$$
\noindent
as argued in Proposition 2.2.4.1 of \cite{lurie2009} ($\mathbf{St}$ being the straightening functor).
It follows that $\mathcal{C}_{/x} \to \mathcal{C}$ is fiberwise equivalent (hence contravariantly equivalent) to the fibrant-cofibrant replacement of $F$.

Therefore, the essential image through $g$ of the subcategory of fibrant-cofibrant replacements of the representable functors is a simplicial category equivalent to $\overline{\mathcal{C}}_{\mathbf{fb}}$.

Secondly, the morphism $f: \mathbf{N}_\Delta((\mathbf{SSet}^{\mathfrak{C}(\mathcal{C}^{op})})^\circ) \rightarrow Fun(\mathcal{C}^{op}, \mathcal{S})$ is obtained, by the Yoneda lemma, from the map (natural in $\mathcal{D}$)
$$Hom_{\mathbf{Cat_\Delta}}(\mathcal{D}, (\mathbf{SSet}^{\mathfrak{C}(\mathcal{C}^{op})})^\circ) \to Hom_{\mathbf{Cat_\Delta}}(\mathcal{D} \times \mathfrak{C}(\mathcal{C}^{op}),(\mathbf{SSet})^\circ)$$

\noindent resulting from the evaluation map $(\mathbf{SSet}^{\mathfrak{C}(\mathcal{C}^{op})})^\circ \times \mathfrak{C}(\mathcal{C}^{op}) \to (\mathbf{SSet})^\circ$.

When restricted to the (fibrant-cofibrant replacements of) representable simplicial functors, the image of $f$ is constituted of functors $\mathcal{C}^{op} \to \mathcal{S}$ that are essentially representable, in the sense that they are in the essential image of the Yoneda embedding (by the very definition of the latter).

Therefore, when restricting to the simplicial functors which are representable up to equivalence (this is indicated by the $-_{|repr}$ subscript below), the zig-zags

\adjustbox{scale=0.85}{%
\begin{minipage}{\linewidth}
\begin{align*} \label{zigzag}
& \mathcal{C} \rightarrow Fun(\mathcal{C}^{op}, \mathcal{S})_{|repr} \leftarrow \mathbf{N}_\Delta((\mathbf{SSet}^{\mathfrak{C}(\mathcal{C}^{op})})^\circ)_{|repr} \leftarrow \mathbf{N}_\Delta(\overline{\mathcal{C}})\\
& \mathcal{C} \rightarrow Fun(\mathcal{C}^{op}, \mathcal{S})_{|repr} \leftarrow \mathbf{N}_\Delta((\mathbf{SSet}^{\mathfrak{C}(\mathcal{C}^{op})})^\circ)_{|repr} \rightarrow \mathbf{N}_\Delta(\mathbf{RFib}(\mathcal{C}))_{|repr} \leftarrow  \mathbf{N}_\Delta(\overline{\mathcal{C}}_{\mathbf{fb}})
\end{align*}
\end{minipage}}

\vspace{2mm}

\noindent constitute zig-zags of equivalences. This concludes the proof.\end{proof}

\section{Adding structure}
\label{s2}

In this section, we additionally suppose that $\mathcal{C}$ is locally cartesian closed.
We aim to rigidify the the structure of $\mathcal{C}$, that is, to carry over structure, like (chosen) pullbacks that are coherent up to equivalences in $\mathcal{C}$, to "rigid" structure, consisting of (chosen) pullbacks in $\overline{\mathcal{C}}$ that are coherent up to isomorphisms. It is to be understood that such a level of coherence allows one to connect back with the strict features of type theory, relying on some splitting procedure, such as mentioned in the introduction, so that coherence up to isomorphisms is turned into strict coherence (i.e, up to equalities/on the nose).

From now on, the only model structure we consider on $\mathbf{SSet}^{\mathfrak{C}(\mathcal{C}^{op})}$ is the global injective one.

\begin{definition}
 
 \label{cstar}
 We define $\overline{\mathcal{C}}^*$ as the smallest full subcategory of $\mathbf{SSet}^{\mathfrak{C}(\mathcal{C}^{op})}$ containing $\overline{\mathcal{C}}$ and stable under the following operations:
 \begin{itemize}
  \item Taking pullbacks along fibrations
  \item Forming the objects of the form $\Pi_f(x)$ where $f$ and $x$ are fibrations and $\Pi_f$ the right adjoint to the pullback functor along $f$
  \item Forming the (cofibration,trivial fibration) and (trivial cofibration, fibration) factorizations of a morphism.
 \end{itemize}
\end{definition}

We claim the following:

\begin{proposition}

 \label{strict_struct} 
 $\mathcal{C}$ is equivalent to the simplicial nerve of $\overline{\mathcal{C}}^*$.%
 
\end{proposition}

The idea of the proof is that the strict structure added by Definition \ref{cstar} is already present up to equivalence, in the sense that, for instance, the pullback of a cospan with objects of the forms $Hom(-,\overline{y}(x_i))$ should be equivalent to $Hom(-,\overline{y}(p))$, where $p$ is the pullback of a corresponding diagram in $\mathcal{C}$, since homotopy pullbacks in $\mathbf{SSet}^{\mathfrak{C}(\mathcal{C}^{op})}$ yield pullbacks when applying the simplicial nerve.
We record the latter observation in \cref{hol_lemma} below.

Given a simplicial model category $\mathcal{M}$, we note $\mathcal{M}^\circ$ for the full simplicial subcategory spanned by the fibrant-cofibrant objects. Write $\mathbf{Csp}$ for the "cospan-shaped" category. We consider diagrams $X : \mathbf{Csp} \to \mathcal{M}^\circ$ and $X' : \mathbf{Csp} \to \mathbf{N_\Delta}(\mathcal{M}^\circ)$ that are equivalent, in that there are isomorphic as objects of the homotopy category $\mathbf{Ho}({\mathcal{M}^\circ}^\mathbf{Csp}) \simeq \mathbf{Ho}(\mathbf{N_\Delta}({\mathcal{M}^\circ}^\mathbf{Csp}))$.
 
 Note that the latter "identity-on-objects" isomorphism stems from the fact that $\mathcal{M}$ is simplicial, so that the hom-spaces computed from the simplicial enrichment in $\mathcal{M}^\circ$ coincide, up to weak equivalences, with the ones computed in the quasicategory $\mathbf{N_\Delta}(\mathcal{M}^\circ)$. In particular, the induced sets of morphisms (considered up to homotopy) between two $\mathbf{Csp}$-indexed diagrams are the same.

\begin{lemma} 
\label{hol_lemma}
 With the previous notations, the homotopy limit of $X$ (computed, for instance, as a limit of a fibrant replacement of $X$ for the Reedy model structure) is isomorphic to the limit of $X'$ (in the sense of quasicategories), as objects of the homotopy category $\mathbf{Ho}(\mathcal{M}^\circ) \simeq \mathbf{Ho}(\mathbf{N_\Delta}(\mathcal{M}^\circ))$.
\end{lemma}

\begin{proof}
 The diagonal functor 
 
 $$\Delta : \mathcal{M} \to \mathcal{M}^\mathbf{Csp}$$ is well-known to be left Quillen when $\mathcal{M}^\mathbf{Sp}$ is endowed with the usual Reedy model structure (the one such that the fibrant object are the cospan where all objects are fibrant, and at least one of the arrows is a fibration). 
 Hence, the adjunction defining the limit of a cospan diagram yields a derived adjunction, where the left adjoint
 
 $$\mathbf{Ho}(\mathcal{M}^\circ) \to \mathbf{Ho}({\mathcal{M}^\circ}^\mathbf{Csp})$$
 is also a derived functor of the diagonal ($\infty$-)functor 
 
 $$\mathbf{N_\Delta}(\Delta) : \mathbf{N_\Delta}(\mathcal{M}^\circ) \to \mathbf{N_\Delta}({\mathcal{M}^\circ}^\mathbf{Csp})$$
 (this follows from the explicit construction of the latter).
 
 Therefore, their right adjoints also coincide, yielding isomorphic objects when applied to isomorphic diagrams.
 \end{proof}
 
 A similar statement holds for dependent products:
 
 \begin{lemma}
 \label{dp_lemma}
  Let $\mathcal{M}$ be a locally cartesian closed model category. Consider fibrations between fibrant objects $f : b \to c$ and $g : a \to b$ in $\mathcal{M}$. Write $\tilde{f} : b \to c$ and $\tilde{g} : a \to b$ for the corresponding $1$-simplices of the locally cartesian closed quasicategory $\tilde{\mathcal{M}}$ presented by $\mathcal{M}$. Then the dependent product $\Pi_f g$ is isomorphic to the dependent product $\Pi_{\tilde{f}} \tilde{g}$ (computed in the quasicategory) as objects of $\mathbf{Ho}(\mathcal{M}_{/c}) \simeq \mathbf{Ho}(\tilde{\mathcal{M}}_{/c})$.
 \end{lemma}
 
 \begin{proof}
  The isomorphism $\mathbf{Ho}(\mathcal{M}_{/c}) \simeq \mathbf{Ho}(\tilde{\mathcal{M}}_{/c})$, telling that the slice model category above a fibrant object models the corresponding slice of the underlying quasicategory, follows from Corollary 7.6.13 in \cite{cisinski2019}.
  
  The functors that postcompose with $f$ (or $\tilde{f}$) fit in commutative square:

\[\begin{tikzcd}[ampersand replacement=\&]
	{\mathbf{Ho}(\mathcal{M}_{/b})} \&\& {\mathbf{Ho}(\mathcal{M}_{/c})} \\
	\\
	{\mathbf{Ho}(\tilde{\mathcal{M}}_{/b})} \&\& {\mathbf{Ho}(\tilde{\mathcal{M}}_{/c})}
	\arrow["{\mathbf{Ho}(f_!)}", from=1-1, to=1-3]
	\arrow["\simeq"{description}, from=1-1, to=3-1]
	\arrow["\simeq"{description}, from=1-3, to=3-3]
	\arrow["{\mathbf{Ho}(\tilde{f}_!)}"', from=3-1, to=3-3]
\end{tikzcd}\]

From that, we deduce that a similar square for the right adjoint $f^*$ (and $\tilde{f}^*$) commutes up to isomorphism, and, ultimately, likewise for the further right adjoint $f_*$ (and $\tilde{f}_*$):

\[\begin{tikzcd}[ampersand replacement=\&]
	{\mathbf{Ho}(\mathcal{M}_{/b})} \&\& {\mathbf{Ho}(\mathcal{M}_{/c})} \\
	\\
	{\mathbf{Ho}(\tilde{\mathcal{M}}_{/b})} \&\& {\mathbf{Ho}(\tilde{\mathcal{M}}_{/c})}
	\arrow["{\mathbf{Ho}(f_*)}", from=1-1, to=1-3]
	\arrow["\simeq"{description}, from=1-1, to=3-1]
	\arrow["\simeq"{description}, draw=none, from=1-1, to=3-3]
	\arrow["\simeq"{description}, from=1-3, to=3-3]
	\arrow["{\mathbf{Ho}(\tilde{f}_*)}"', from=3-1, to=3-3]
\end{tikzcd}\]

In particular, starting with $g \in \mathcal{M}_{/b}$ (and  $\tilde{g} \in \tilde{\mathcal{M}}_{/b}$), the results of applying the functor $f_*$ (resp. $\tilde{f}_*$) are the same up to isomorphism.  
 \end{proof}

\begin{proof}[Proof of Proposition \ref{strict_struct}]

 Observe that $\overline{\mathcal{C}}^*$ is the colimit of the tower $$\overline{\mathcal{C}}^0 \to ... \to \overline{\mathcal{C}}^n \to ...$$ where $\overline{\mathcal{C}}^0 = \overline{\mathcal{C}}$ and $\overline{\mathcal{C}}^{i+1}$ is the full subcategory of $\mathbf{SSet}^{\mathfrak{C}(\mathcal{C}^{op})}$ obtained by adjoining to $\overline{\mathcal{C}}^i$ the objects of the form $f^*(a)$, $\Pi_f g$ for $f$ and $g$ fibrations, and $a$ an object in $\overline{\mathcal{C}}^i$, as well as adjoining the object appearing in the (cofibration,trivial fibration) and (trivial cofibration, fibration) factorizations of a morphism $f$ in $\overline{\mathcal{C}}^i$. 
 This is because, since the three operations we consider only involve finitely many objects and arrows, we only need $\omega$ steps for the construction to stabilize.  
 $\overline{\mathcal{C}}^*$ is therefore the full subcategory of $\mathbf{SSet}^{\mathfrak{C}(\mathcal{C}^{op})}$  such that $x$ is an object of $\overline{\mathcal{C}}^*$ if and only if it is an object of any of the $\overline{\mathcal{C}}^i$.
 
 It will be enough to prove that the canonical inclusion $\iota: \overline{\mathcal{C}} \to \overline{\mathcal{C}}^*$ is an equivalence of simplicially enriched categories. 
 Since this functor is by definition full and faithful, it will be enough to check that all the objects of $\overline{\mathcal{C}}^*$ are equivalent to an object of $\overline{\mathcal{C}}$.
 
 \noindent
 If $a$ is of the form $\overline{y}(x)$ for $x$ an object of $\overline{\mathcal{C}}$, it is already in the image of $\iota$.
 
 Suppose that, for a natural number $i$, we have proved that every object of $\overline{\mathcal{C}}^i$ is equivalent to an object of $\overline{\mathcal{C}}$ (this is tautological when $i = 0$).
 Then any cospan $u$, with limit $p$ in $\overline{\mathcal{C}}^i$, of the form 
 
\[\begin{tikzcd}
	&& C \\
	\\
	B && A
	\arrow["F"', from=3-1, to=3-3]
	\arrow["G", from=1-3, to=3-3]
\end{tikzcd}\]

\noindent
where $F$ is a fibration, is equivalent to a cospan $u'$

\[\begin{tikzcd}
	&& {\overline{y}(c)} \\
	\\
	{\overline{y}(b)} && {\overline{y}(a)}
	\arrow["f"', from=3-1, to=3-3]
	\arrow["g", from=1-3, to=3-3]
\end{tikzcd}\]

Now, taking $f'$ and $g'$ to be morphisms in $\mathcal{C}$ obtained by pulling/pushing the morphisms $f$ and $g$ (seen as a edges of the simplicial nerve of $\overline{\mathcal{C}}$) along the zig-zag establishing the equivalence between $\mathcal{C}$ and $\mathbf{N}_\Delta(\overline{\mathcal{C}})$, we get a cospan $u''$ in $\mathcal{C}$.

Note $d$ for its limit in $\mathcal{C}$.
Since the Yoneda embedding and the equivalences of $\infty$-categories preserve limits, $\overline{y}(d)$ is a limit of the cospan $u'$ in $\mathbf{N}_\Delta({\mathbf{SSet}^{\mathfrak{C}(\mathcal{C}^{op})}}^\circ)$. The limit $p$ of $u$ computed in $\mathbf{SSet}^{\mathfrak{C}(\mathcal{C}^{op})}$ is a homotopy limit, so it yields a limit in the quasicategory $\mathbf{N}_\Delta({\mathbf{SSet}^{\mathfrak{C}(\mathcal{C}^{op})}}^\circ)$ by Lemma \ref{hol_lemma}. Therefore, $p$ and $\overline{y}(d)$ are seen to be equivalent, as limits of equivalent cospans.

A similar reasoning applies to the dependent product along a fibration $\Pi_f g$, since dependent products are also preserved by the Yoneda embedding, and since the defining adjunction yields a derived adjunction (and even an adjunction of $\infty$-categories), given that the functor $\Pi_f$ is right Quillen. Explicitly, if $f$ and $g$ are equivalent to $\overline{y}(\tilde{f})$ and $\overline{y}(\tilde{g})$ respectively, then $\overline{y}(\Pi_{\tilde{f}} \tilde{g})$ is equivalent to $\Pi_f g$ by Lemma \ref{dp_lemma}.

The (cofibration,trivial fibration) and (trivial cofibration, fibration) factorizations case is evident.

This proves the result for $i+1$ and concludes the proof.\end{proof}

\begin{remark}
\label{size_issue}
Note that $\overline{\mathcal{C}}^*$ is a small subcategory of $\mathbf{SSet}^{\mathfrak{C}(\mathcal{C}^{op})}$ by construction, so that there exists a cardinal $\kappa$ such that all objects of $\overline{\mathcal{C}}^*$ are $\kappa$-presentable.

 \end{remark}

 \begin{remark}
  In this section and the following ones, we will not be explicitly concerned with the coherence problems arising from the fact that pullbacks (that is, substitutions) only preserve the structure in a weak sense. Therefore, it will be implicit that we apply a splitting operation to the comprehension category modeling the type theory we are working with (precisely, the left-splitting introduced in \cite{lw2015local}).
 \end{remark}
 
 \begin{theorem}
 
 \label{t1}
 
 $\overline{\mathcal{C}}^*$ yields an interpretation of MLTT. %
 \end{theorem}
 
 \begin{proof}
  The simplicial category $\overline{\mathcal{C}}^*$ actually forms a $\pi$-tribe, as defined in \cite{joyal2017notes}, which can be used as the basis of a (non-split) comprehension category $(\overline{\mathcal{C}}^*)^\rightarrow_\text{fib} \to \overline{\mathcal{C}}^*$, where $(\overline{\mathcal{C}}^*)^\rightarrow_\text{fib}$ is the full subcategory of $(\overline{\mathcal{C}}^*)^\rightarrow$ spanned by the fibrations.
  
  To see this, it is convenient to picture $\overline{\mathcal{C}}^*$ as a sub-tribe of the simplicial $\pi$-tribe $R\mathcal{C}$ of essentially representable presheaves in $\mathbf{SSet}^{\mathfrak{C}(\mathcal{C}^{op})}$. The former is indeed closed under pullbacks along fibrations, and under taking the factorization of a morphism as an anodyne (actually a trivial cofibration) followed by a fibration. We further observe that $R\mathcal{C}$ is a simplicial $\pi$-tribe, hence it is equivalent to its simplicial localization as a tribe, and the inclusion $\overline{\mathcal{C}}^* \to R\mathcal{C}$ is a weak equivalence of tribes, as one can be seen by checking that it induces an equivalence of categories between the homotopy categories $\mathbf{Ho}(\overline{\mathcal{C}}^*) \to \mathbf{Ho}(R\mathcal{C})$. Therefore, the $(\infty,1)$-category stemming from the tribe structure on $\overline{\mathcal{C}}^*$ is indeed the same as the one stemming from the simplicially enriched category structure.
 \end{proof}

\section{Universes and univalence}
\label{s3}

In this section, we further strengthen our hypothesis by supposing that $\mathcal{C}$ admits a univalent map $\mathbf{\pi} : \tilde{\mathbf{u}} \to \mathbf{u}$ classifying a pullback-stable class of morphisms $\mathbf{S}$.

The goal is to construct an object that provides a universe for $\mathbf{S}$ in an adequate sense. This is more involved than adding, for instance, pullbacks along fibrations because the construction of universes in a model category will be realized relying on some sophisticated technology that has been introduced by Shulman in \cite{shulman2019}.

We start by recalling what it means for a universe to be univalent, first in a quasicategory, then in a model category.
The first definition stems from Theorem 6.28 in \cite{rasekh2018cso} (in a situation where it is formulated in terms of complete Segal spaces rather than quasicategories):

\begin{definition}
 Let $p : e \to b$ be a map in $\mathcal{C}$.
Note $\mathcal{O}_\mathcal{C} := Fun(\Delta_1,\mathcal{C})$ and $\mathcal{O}_\mathcal{C}^{(p)}$ for the subcategory whose objects are the maps that are pullbacks of $p$, and whose morphisms are pullback squares.

There is a canonical map of right fibrations $\alpha : \mathcal{C}_{/b} \to \mathcal{O}_\mathcal{C}^{(p)}$ induced by the Yoneda lemma (this map essentially takes a morphism $x \to b$ to the morphism $p^*e \to x$ obtained by pullback from $p$).
The map $p : e \to b$ is \textit{univalent} when $\alpha$ is an equivalence.

Note that, equivalently, the map $p$ is univalent when $p$ is a final object of $\mathcal{O}_\mathcal{C}^{(p)}$.
\end{definition}

\begin{remark}
 
In particular, $p$ being univalent implies that if a morphism $q$ is displayed in two ways as a pullback of $p$ as in the diagrams below,

\[\begin{tikzcd}
	y && e && y && e \\
	\\
	x && b && x && b
	\arrow["q"', from=1-1, to=3-1]
	\arrow["\phi"', from=3-1, to=3-3]
	\arrow["p", from=1-3, to=3-3]
	\arrow[from=1-1, to=1-3]
	\arrow["q"', from=1-5, to=3-5]
	\arrow[from=1-5, to=1-7]
	\arrow["p", from=1-7, to=3-7]
	\arrow["\psi"', from=3-5, to=3-7]
	\arrow["\lrcorner"{anchor=center, pos=0.125}, draw=none, from=1-5, to=3-7]
	\arrow["\lrcorner"{anchor=center, pos=0.125}, draw=none, from=1-1, to=3-3]
\end{tikzcd}\]

\noindent
then $\phi$ and $\psi$ are homotopic.

\end{remark}

The second definition (or rather characterization), due to Voevodsky (see \cite{slides_gambino} and \cite{slides_joyal}), is the following:

\begin{definition}
\label{univfib}
 In a model category $\mathcal{M}$, a fibration $P : E \to B$ is univalent when, for every fibration $Q : Y \to X$, the pair of maps $(\phi,\phi')$ in the following homotopy pullback diagram is unique up to homotopy, when such a pair exists (that is, the space of such pair of maps is contractible).

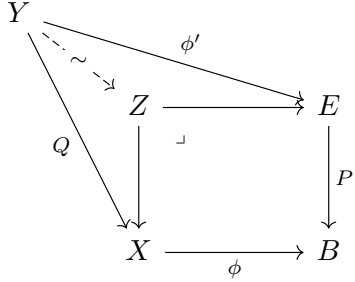
\[\begin{tikzcd}
	Y \\
	& Z && E \\
	\\
	& X && B
	\arrow["Q"', from=1-1, to=4-2]
	\arrow["\phi"', from=4-2, to=4-4]
	\arrow["P", from=2-4, to=4-4]
	\arrow["{\phi'}", from=1-1, to=2-4]
	\arrow[from=2-2, to=2-4]
	\arrow[from=2-2, to=4-2]
	\arrow["\lrcorner"{anchor=center, pos=0.125}, draw=none, from=2-2, to=4-4]
	\arrow["\sim"{description}, dashed, from=1-1, to=2-2]
\end{tikzcd}\]

\noindent
Here the dotted map is a weak equivalence (by definition of homotopy pullbacks).
\end{definition}

Our goal is now the following:

\begin{proposition}
 There exists a univalent universe $\mathscr{U}$ in $\mathbf{SSet}^{\mathfrak{C}(\mathcal{C}^{op})}$ classifying the fibrations equivalent to a morphism corresponding to an edge in $\mathbf{S}$.
 
 Moreover, $\mathscr{U}$ is essentially representable, in the sense that it is equivalent to an object of $\overline{\mathcal{C}}$.
\end{proposition}

Here, the notion of equivalence between morphisms in $\mathbf{SSet}^{\mathfrak{C}(\mathcal{C}^{op})}$ refers to isomorphisms in the homotopy category $Ho((\mathbf{SSet}^{\mathfrak{C}(\mathcal{C}^{op})})^\rightarrow)$ (alternatively, this can be seen as the existence of a zig-zag of squares between two morphisms pictured vertically, where all the horizontal arrows are equivalence). 

We will deduce this result directly from Theorem 5.22 in \cite{shulman2019}, by constructing an appropriate notion of fibred structure (as introduced in section 3 of that paper, which we recall below) in:
$$\mathcal{E} := \mathbf{SSet}^{\mathfrak{C}(\mathcal{C}^{op})}$$
Our starting point for a universal fibration is the following naive candidate:

\begin{definition}
 Define $\pi: \tilde{U} \to U$ in $\mathcal{E}$ to be a fibration between fibrant objects corresponding to the universal map $\tilde{\mathbf{u}} \to \mathbf{u}$ in $\mathcal{C}$. This means that $\pi$ is equivalent to the morphism $Hom(-,\tilde{\mathbf{u}}) \to Hom(-,\mathbf{u})$ in the previous sense. %
\end{definition}

In a locally cartesian closed category, it is possible to construct an object \textit{representing} the isomorphism between two objects (in the sense that, for instance, the object $\mathbf{Iso}(A,B)$ has global elements which are naturally in bijection with the set of isomorphisms between $A$ and $B$).

In a suitable model category, a similar construction, though more sophisticated, can be carried out to represent, not just isomorphisms, but (structured) weak equivalences in general. The construction of such an object is detailed in Section 4 of \cite{shulman2013}. We record the corresponding result in the lemma below.

\begin{definition}
 
 A \textit{coherent simplicial homotopy equivalence} in a simplicial model category is the data of a morphism $f : a \to b$, a left inverse (up to homotopy) $g$, a right inverse $h$, as well as the homotopies $g \circ f \sim id_a$ and $f \circ h \sim id_b$.  
 
\end{definition}

\begin{lemma}
 Consider two fibrations $E_1 \to B$ and $E_2 \to B$ in $\mathcal{E}$.
 There is an object $\text{Equiv}_B(E_1,E_2)$ over $B$ such that for any object Y and maps $f: Y \to B$ , morphisms from  $f$ to $\text{Equiv}_B(E_1,E_2)$ over $B$ are naturally in bijection with coherent simplicial homotopy equivalence $f^*E_1 \to f^*E_2$ over $Y$.
\end{lemma}

The situation can be pictured as in the following diagram:

\[\begin{tikzcd}
	& {f^*E_2} && {E_2} \\
	{f^*E_1} && {E_1} \\
	&&&&& {\text{Equiv}_B(E_1,E_2)} \\
	& Y && B
	\arrow[two heads, from=2-3, to=4-4]
	\arrow[two heads, from=1-4, to=4-4]
	\arrow["f"', from=4-2, to=4-4]
	\arrow[two heads, from=2-1, to=4-2]
	\arrow[two heads, from=1-2, to=4-2]
	\arrow[two heads, from=3-6, to=4-4]
	\arrow[dashed, from=4-2, to=3-6]
	\arrow[dashed, from=2-1, to=1-2]
	\arrow[from=2-1, to=2-3]
	\arrow[from=1-2, to=1-4]
	\arrow["\lrcorner"{anchor=center, pos=0.125}, draw=none, from=2-1, to=4-4]
	\arrow["\lrcorner"{anchor=center, pos=0.125}, draw=none, from=1-2, to=4-4]
\end{tikzcd}\]

\begin{definition}
 We define a (large) groupoid-valued pseudofunctor $\mathbf{Rh}_\pi: \mathcal{E}^{op} \to \catname{GPD}$ as follows:
 
 \begin{itemize}
  \item For $X$ an object,  $\mathbf{Rh}_\pi(X)$ has as its underlying (large) set of objects the set of diagrams of this form:

\begin{equation}
 \label{hcd}
\begin{tikzcd}
	& Z && {\tilde{U}} \\
	Y \\
	& X && U
	\arrow["\pi", two heads, from=1-4, to=3-4]
	\arrow[from=3-2, to=3-4]
	\arrow[two heads, from=1-2, to=3-2]
	\arrow[from=1-2, to=1-4]
	\arrow["h", dashed, from=2-1, to=1-2]
	\arrow["\lrcorner"{anchor=center, pos=0.125}, draw=none, from=1-2, to=3-4]
	\arrow[from=2-1, to=3-2]
\end{tikzcd}
\end{equation}

where $h$ is a coherent simplicial homotopy equivalence. The morphisms are the (natural) isomorphisms between two such diagrams, above the map $\pi: \tilde{U} \to U$.

We will refer to such diagrams as \textit{homotopy classification diagrams}.

  \item For $X' \to X$ a morphism, and a homotopy classification diagram for $X$, we construct one for $X'$ as follows:

\[\begin{tikzcd}
	& Z' && Z && {\tilde{U}} \\
	{Y'} && Y \\
	& {X'} && X && U
	\arrow[two heads, from=1-4, to=3-4]
	\arrow["h", dashed, from=2-3, to=1-4]
	\arrow[from=1-2, to=1-4]
	\arrow["\lrcorner"{anchor=center, pos=0.125}, draw=none, from=1-2, to=3-4]
	\arrow[from=1-4, to=1-6]
	\arrow[two heads, from=1-6, to=3-6]
	\arrow[from=3-4, to=3-6]
	\arrow[from=2-3, to=3-4]
	\arrow[from=3-2, to=3-4]
	\arrow["\lrcorner"{anchor=center, pos=0.125}, draw=none, from=1-4, to=3-6]
	\arrow[two heads, from=1-2, to=3-2]
	\arrow[from=2-1, to=2-3]
	\arrow[from=2-1, to=3-2]
	\arrow["\lrcorner"{anchor=center, pos=0.125, rotate=45}, draw=none, from=2-1, to=3-4]
	\arrow["{h'}", dashed, from=2-1, to=1-2]
\end{tikzcd}\]

Clearly $h'$ inherits a coherent simplicial homotopy equivalence structure from such a structure on $h$.

\end{itemize}

This construction is indeed pseudofunctorial (similarly to the situation where one considers ordinary classification diagrams, see Example 3.9 in \cite{shulman2019}).

\end{definition}

\begin{remark}
 If $F : E_0 \to B$ and $G : E_1 \to B$ are two fibrations, a morphism $F \to G$ (in the slice over $B$) is weak equivalence if and only if it has a coherent simplicial homotopy equivalence structure. This is because $F$ and $G$ are fibrant-cofibrant objects for the model structure on the slice.
\end{remark}

We briefly recall the following definitions from \cite{shulman2019}:

\begin{definition}
 Write $\mathcal{\mathbf{PSH}}(\mathcal{E})$ for the 2-category of contravariant pseudofunctors from $\mathcal{E}$ to $\catname{GPD}$, with pseudonatural transformations between them.
 
 A strict discrete fibration in $\mathcal{\mathbf{PSH}}(\mathcal{E})$ is a strictly natural transformation $\mathds{X} \to \mathds{Y}$ such that each component $\mathds{X}(A) \to \mathds{Y}(A)$ is a discrete fibration.

 The core self-indexing of $\mathcal{E}$ is defined to be the functor $\mathds{E}: \mathcal{E}^{op} \to \catname{GPD}$ mapping $Y$ to the core of $\mathcal{E}_{/Y}$.
 
 \end{definition}
 
 \begin{definition}

 A notion of fibred structure (on $\mathcal{E}$) is a strict discrete fibration $\phi: \mathds{F} \to \mathds{E}$ with small fibers (where $\mathds{F}$ is an object of $\mathcal{\mathbf{PSH}}(\mathcal{E})$ and $\mathds{E}$ is the core self-indexing defined previously).
 
 Given such a fibration, an $\mathds{F}$-structure on a morphism $f: X \to Y$ in $\mathcal{E}$ is an element of the fiber of $\phi(Y)$ above $f$. When $f$ is equipped with an ~$\mathds{F}$-structure, it is called an $\mathds{F}$-algebra. An $\mathds{F}$-morphism from $f$ to $f'$ is a pullback square (exhibiting $f'$ as a pullback of $f$) such that $f'$ is equipped with the $\mathds{F}$-algebra structure, induced from that of $f$ thanks to the function between fibers given by $\phi$.
 \end{definition}

Assigning to a homotopy classification diagram (\ref{hcd}) for $X$ the morphism $Y \to X$ allows us to define (up to the strictification process for a fibration discussed in Section 2 of \cite{shulman2019}) a notion of fibred structure $\phi: \mathbf{Rh}_\pi \to \mathds{E}$.%

There are three important properties for a notion of fibred structure to verify for constructing a corresponding universe from Theorem 5.22 in \cite{shulman2019}. We recall the meaning of these properties (Definitions 3.8, 3.9 and 3.10), then give the corresponding proof for $\mathbf{Rh}_\pi \times_{\mathds{E}} \mathbf{Fib}$, where $\mathbf{Fib}$ is a (well-behaved) notion of fibred structure for the class of fibrations. %

\begin{definition}
 A notion of fibred structure $\phi: \mathds{F} \to \mathds{E}$ is called locally representable if the pullback of representable pseudofunctors along $\phi$ is representable. Concretely, this means that, given any map $X \to Z$ in $\mathcal{E}$ (seen by Yoneda as morphism of $Hom(-,X) \to \mathds{E}(Z)$), there is a map $\phi_X^{\mathds{F}}: \mathds{F}_X \to Z$ such that, for any $g: Y \to Z$, there is a natural bijection between $\mathds{F}$-structures on  $g^*X$ and lifts of $g$ to $\mathds{F}_X$.
\end{definition}

\begin{lemma}
 $\phi: \mathbf{Rh}_\pi \times_{\mathds{E}} \mathbf{Fib} \to \mathds{E}$ is a locally representable notion of fibred structure.
\end{lemma}

\begin{proof}
 Since $\mathbf{Fib}$ is chosen to be locally representable (Corollary 8.29 in \cite{shulman2019}), it is enough to show that so is $\mathbf{Rh}_\pi$ (by Example 3.11 in \cite{shulman2019}).

 The proof follows directly from Example 3.15 of (\cite{shulman2019}), replacing the object of isomorphisms with an object of weak equivalences.
 
 Actually, the notion of coherent simplicial homotopy equivalence is such that $\text{Equiv}_{Z \times U}(X \times U,Z \times \tilde{U})$ represents $\mathbf{Rh}_{\pi,X}$ (when looking at maps that are already known to be fibrations).
 \end{proof}

\begin{definition}
 A notion of fibred structure $\phi: \mathds{F} \to \mathds{E}$ is relatively acyclic if, for any pullback as in the diagram below,

\[\begin{tikzcd}
	{X'} && X \\
	\\
	{Y'} && Y
	\arrow["{f'}"', from=1-1, to=3-1]
	\arrow["f", from=1-3, to=3-3]
	\arrow["g", from=1-1, to=1-3]
	\arrow["i"', hook, from=3-1, to=3-3]
	\arrow["\lrcorner"{anchor=center, pos=0.125}, draw=none, from=1-1, to=3-3]
\end{tikzcd}\]

\noindent
where $i$ is a cofibration and $f$ and $f'$ are $\mathds{F}$-algebras, there exists a new ~$\mathds{F}$-structure on $f'$ making the diagram an $\mathds{F}$-morphism.
\end{definition}

 \begin{lemma}
 $\phi: \mathbf{Rh}_\pi \times_{\mathds{E}} \mathbf{Fib} \to \mathds{E}$ is a relatively acyclic notion of fibred structure.
 \end{lemma}

 \begin{proof}

 Consider a pullback diagram 
 
\[\begin{tikzcd}
	{X'} && X \\
	\\
	{Y'} && Y
	\arrow["F", two heads, from=1-3, to=3-3]
	\arrow["I"', hook, from=3-1, to=3-3]
	\arrow["{F'}"', two heads, from=1-1, to=3-1]
	\arrow["G", from=1-1, to=1-3]
	\arrow["\lrcorner"{anchor=center, pos=0.125}, draw=none, from=1-1, to=3-3]
\end{tikzcd}\]

\noindent
where $F$ and $F'$ are $\mathbf{Rh}_\pi \times_{\mathds{E}} \mathbf{Fib}-$algebras and $i$ is a cofibration.

The pullback diagram we considered is a homotopy pullback (because $\mathbf{SSet}^{\mathfrak{C}(\mathcal{C}^{op})}$ is right proper), so it corresponds in $Fun(\mathcal{C}^{op}, \mathcal{S})$ to a pullback:

\[\begin{tikzcd}
	{x'} && x \\
	\\
	{y'} && y
	\arrow["f", from=1-3, to=3-3]
	\arrow["i"', from=3-1, to=3-3]
	\arrow["{f'}"', from=1-1, to=3-1]
	\arrow["g", from=1-1, to=1-3]
	\arrow["\lrcorner"{anchor=center, pos=0.125}, draw=none, from=1-1, to=3-3]
\end{tikzcd}\]

\noindent
Likewise, the homotopy classification diagrams for $F$ and $F'$ correspond to pullbacks in $Fun(\mathcal{C}^{op}, \mathcal{S})$:

\[\begin{tikzcd}
	{x'} && {\tilde{\mathbf{u}}} && x \\
	\\
	{y'} && {\mathbf{u}} && y
	\arrow[from=1-3, to=3-3]
	\arrow["{h'}"', from=3-1, to=3-3]
	\arrow["{f'}"', from=1-1, to=3-1]
	\arrow[from=1-1, to=1-3]
	\arrow["\lrcorner"{anchor=center, pos=0.125}, draw=none, from=1-1, to=3-3]
	\arrow[from=1-5, to=1-3]
	\arrow["h", from=3-5, to=3-3]
	\arrow["f", from=1-5, to=3-5]
	\arrow["\lrcorner"{anchor=center, pos=0.125, rotate=-90}, draw=none, from=1-5, to=3-3]
\end{tikzcd}\]

Univalence of $\mathbf{u}$ in $\mathcal{C}$ implies the existence of a homotopy (in the sense given by the global injective model structure) from $H \circ I$ to $H'$ (where $H$ and $H'$ are morphisms in $\mathcal{E}$ corresponding to $h$ and $h'$). 

Since $I$ is a cofibration, the homotopy extension property (that is, the right lifting property of a trivial fibration projection $\mathbf{Path}(U) \to U$ against $I$) implies the existence of $H''$ such that $H' = H'' \circ I$.

\[\begin{tikzcd}[ampersand replacement=\&]
	{Y'} \&\& {\mathbf{Path}(U)} \\
	\\
	Y \&\& U
	\arrow[from=1-1, to=1-3]
	\arrow["I"', from=1-1, to=3-1]
	\arrow[from=1-3, to=3-3]
	\arrow["K", dashed, from=3-1, to=1-3]
	\arrow["H"', from=3-1, to=3-3]
\end{tikzcd}\]

Consider also a homotopy between $H'' \circ F$ and $H \circ F$, and take a solution $L$ of the following lifting problem,

\[\begin{tikzcd}[ampersand replacement=\&]
	X \& Y \& {\tilde{U}} \\
	\\
	{\mathbf{Cyl}(X)} \&\& U
	\arrow[from=1-1, to=1-2]
	\arrow[from=1-1, to=3-1]
	\arrow["{H''}", from=1-2, to=1-3]
	\arrow[two heads, from=1-3, to=3-3]
	\arrow["L", dashed, from=3-1, to=1-3]
	\arrow[from=3-1, to=3-3]
\end{tikzcd}\]
so that the second component of $L$ provides the outer square
\[\begin{tikzcd}[ampersand replacement=\&]
	X \\
	\& {Z''} \&\& {\tilde{U}} \\
	\\
	\& Y \&\& U
	\arrow["\sim"{description}, dashed, from=1-1, to=2-2]
	\arrow[curve={height=-6pt}, from=1-1, to=2-4]
	\arrow[curve={height=6pt}, from=1-1, to=4-2]
	\arrow[from=2-2, to=2-4]
	\arrow[from=2-2, to=4-2]
	\arrow[from=2-4, to=4-4]
	\arrow["{H''}"', from=4-2, to=4-4]
\end{tikzcd}\]
that is also a homotopy pullback, thus the mediating dotted arrow is a weak equivalence.

By Lemma 2.1 in \cite{shulman2013}, this weak equivalence can moreover be promoted to a coherent homotopy equivalence as to induce the one that is part of the $\mathbf{Rh}_\pi$-structure on $F'$. Since $\mathbf{Fib}$ is known to be a relatively acyclic notion of fibred structure, there is also a new $\mathbf{Fib}-$structure on $F$ inducing that on $F'$. Together with $H''$, this provides a new $\mathbf{Rh}_\pi \times_{\mathds{E}} \mathbf{Fib}$-structure on $F$ inducing that on $F'$.
\end{proof}

\begin{definition}
 A notion of fibred structure $\phi: \mathds{F} \to \mathds{E}$ is homotopy invariant if every $\mathds{F}$-algebra is a fibration, and, for any diagram, 

\[\begin{tikzcd}
	{X'} && X \\
	\\
	{Y'} && Y
	\arrow["{f'}"', two heads, from=1-1, to=3-1]
	\arrow["f", two heads, from=1-3, to=3-3]
	\arrow["\sim", from=1-1, to=1-3]
	\arrow["\sim"', from=3-1, to=3-3]
\end{tikzcd}\]

\noindent
where $f$ and $f'$ are fibrations and the horizontal maps are weak equivalences, $f$ can be equipped with an $\mathds{F}$-structure if and only so can $f'$ be. 
\end{definition}

 \begin{lemma}
 $\phi: \mathbf{Rh}_\pi \times_{\mathds{E}} \mathbf{Fib} \to \mathds{E}$ is a homotopy invariant notion of fibred structure.
 \end{lemma}
 
 \begin{proof}
Consider the following diagram, where $f$ and $f'$ are fibrations and the horizontal maps are weak equivalences:

\[\begin{tikzcd}
	{X'} && X \\
	\\
	{Y'} && Y
	\arrow["{f'}"', two heads, from=1-1, to=3-1]
	\arrow["\sim"', from=3-1, to=3-3]
	\arrow["f", two heads, from=1-3, to=3-3]
	\arrow["\sim", from=1-1, to=1-3]
\end{tikzcd}\]

Suppose that $f$ has moreover $\mathbf{Rh}_\pi$-structure. We must prove that $f'$ also has one.

Now form the following two pullback squares where $Z$ is part of the $\mathbf{Rh}_\pi$-structure on $X \to Y$. 

\[\begin{tikzcd}
	{X'} && {X_f} \\
	& {Z'} && {Z} && {\tilde{U}} \\
	\\
	& {Y'} && {Y} && U
	\arrow[two heads, from=1-3, to=4-4]
	\arrow[from=2-4, to=4-4]
	\arrow[from=4-4, to=4-6]
	\arrow[from=2-6, to=4-6]
	\arrow[from=2-4, to=2-6]
	\arrow["\sim", from=1-3, to=2-4]
	\arrow["\sim"', hook, from=4-2, to=4-4]
	\arrow[from=1-1, to=4-2]
	\arrow[from=2-2, to=4-2]
	\arrow["\sim", from=1-1, to=2-2]
	\arrow["\lrcorner"{anchor=center, pos=0.125}, draw=none, from=2-4, to=4-6]
	\arrow["\lrcorner"{anchor=center, pos=0.125}, draw=none, from=2-2, to=4-4]
	\arrow["\sim", from=2-2, to=2-4]
	\arrow["\sim", from=1-1, to=1-3]
\end{tikzcd}\]

The map $X' \to Z'$ is an equivalence (by the $2$-out-of-$3$ property), so we get an $\mathbf{Rh}_\pi$-structure on $X' \to Y'$.

For the converse, we would like to use a homotopy inverse of morphisms defined by the square in the statement. It would be possible to find such an inverse if $f$ and $f'$ were both fibrant-cofibrant objects (for the injective model structure on the arrow category). 
We therefore consider fibrant replacement $g : X_f \to Y_f$ and $g' : X'_f \to Y'_f$ of $f$ and $f'$ respectively:

\[\begin{tikzcd}[ampersand replacement=\&]
	\&\&\& {X_f'} \&\& {X_f} \\
	{X'} \&\& X \\
	\&\&\& {Y_f'} \&\& {Y_f} \\
	{Y'} \&\& Y
	\arrow["{g'}"', two heads, from=1-4, to=3-4]
	\arrow["g", two heads, from=1-6, to=3-6]
	\arrow["\sim"{description}, hook, from=2-1, to=1-4]
	\arrow["\sim", from=2-1, to=2-3]
	\arrow["{f'}"', two heads, from=2-1, to=4-1]
	\arrow["\sim"{description}, hook, from=2-3, to=1-6]
	\arrow["f", two heads, from=2-3, to=4-3]
	\arrow["\sim"{description}, hook, from=4-1, to=3-4]
	\arrow["\sim"', from=4-1, to=4-3]
	\arrow["\sim"{description}, hook, from=4-3, to=3-6]
\end{tikzcd}\]

Now, $g$ and $g'$ are fibrant-cofibrant objects (for the injective model structure on the arrow category) that are isomorphic in the corresponding homotopy category. Therefore, there exists a weak equivalence $g' \to g$. 

We solve the following lifting problem (in the arrow category), as to endow $g'$ with an $\mathbf{Rh}_\pi$-structure,

\[\begin{tikzcd}[ampersand replacement=\&]
	\&\&\& {\tilde{U}} \\
	{X'} \\
	\&\& {X_f'} \& U \\
	{Y'} \\
	\&\& {Y_f'}
	\arrow[two heads, from=1-4, to=3-4]
	\arrow[from=2-1, to=1-4]
	\arrow["\sim"{description}, hook, from=2-1, to=3-3]
	\arrow["{f'}"', two heads, from=2-1, to=4-1]
	\arrow[dashed, from=3-3, to=1-4]
	\arrow["{g'}"', two heads, from=3-3, to=5-3]
	\arrow[from=4-1, to=3-4]
	\arrow["\sim"{description}, hook, from=4-1, to=5-3]
	\arrow[dashed, from=5-3, to=3-4]
\end{tikzcd}\]

\[\begin{tikzcd}[ampersand replacement=\&]
	{X'} \\
	\& {Z'} \&\& {X_f'} \\
	\&\&\&\& {Z_f'} \&\&\& {\tilde{U}} \\
	\\
	\& {Y'} \\
	\&\&\&\& {Y_f'} \&\&\& U
	\arrow["\sim"{description}, dashed, from=1-1, to=2-2]
	\arrow["\sim"{description}, hook, from=1-1, to=2-4]
	\arrow[curve={height=-30pt}, from=1-1, to=3-8]
	\arrow["{f'}"', two heads, from=1-1, to=5-2]
	\arrow["\sim"{description}, hook, from=2-2, to=3-5]
	\arrow[curve={height=-24pt}, from=2-2, to=3-8]
	\arrow[from=2-2, to=5-2]
	\arrow["\lrcorner"{anchor=center, pos=0.125}, draw=none, from=2-2, to=6-5]
	\arrow[dashed, from=2-4, to=3-5]
	\arrow[from=2-4, to=3-8]
	\arrow["{g'}"', two heads, from=2-4, to=6-5]
	\arrow[from=3-5, to=3-8]
	\arrow[from=3-5, to=6-5]
	\arrow["\lrcorner"{anchor=center, pos=0.125}, draw=none, from=3-5, to=6-8]
	\arrow[two heads, from=3-8, to=6-8]
	\arrow["\sim"{description}, hook, from=5-2, to=6-5]
	\arrow[curve={height=-18pt}, from=5-2, to=6-8]
	\arrow[from=6-5, to=6-8]
\end{tikzcd}\]
where the mediating map $X'_f \to Z'_f$ is seen to be a weak equivalence by the $2$-out-of-$3$ property.

By using the first part of the proof on the composite weak equivalence $f \to g \to g'$, we see that $f$ can be endowed with an $\mathbf{Rh}_\pi$-structure.
\end{proof}

We are now in a position to deduce, by theorem 5.22 in \cite{shulman2019}, that there is a univalent universe $\pi': \tilde{\mathscr{U}} \to \mathscr{U}$ classifying the fibrations that can be equipped with an $\mathbf{Rh}_\pi$-structure (and that are small enough, i.e, that are $\kappa$-presentable as mentioned in Remark \ref{size_issue}). Write $\overline{\mathcal{C}}^{**}$ for the smallest full subcategory of $\mathbf{SSet}_{/\mathcal{C}}$ containing the objects of $\overline{\mathcal{C}}^*$, $\tilde{\mathscr{U}}$ and $\mathscr{U}$, and stable under the appropriate constructions (as in Definition \ref{cstar}).
Those morphisms classified by $\pi'$ hence include the fibrations, between objects in $\overline{\mathcal{C}}^{**}$, that are equivalent to morphisms induced by an edge in $\mathbf{S}$.

\begin{proposition}
 The object $\mathscr{U}$ is equivalent to $U$ (hence the object $\tilde{\mathscr{U}}$ is equivalent to $\tilde{U}$).
\end{proposition}

\begin{proof}

By definition of $\pi'$, there is a pullback square:

\[\begin{tikzcd}
	{\tilde{U}} && {\tilde{\mathscr{U}}} \\
	\\
	U && {\mathscr{U}}
	\arrow["\pi"', from=1-1, to=3-1]
	\arrow["{\pi'}", from=1-3, to=3-3]
	\arrow["\phi"', from=3-1, to=3-3]
	\arrow[from=1-1, to=1-3]
	\arrow["\lrcorner"{anchor=center, pos=0.125}, draw=none, from=1-1, to=3-3]
\end{tikzcd}\]

\noindent
Because $\pi'$ can be equipped with an $\mathbf{Rh}_\pi$-structure, we also have a homotopy classification diagram:

\[\begin{tikzcd}
	{\tilde{\mathscr{U}}} & {Z_{\mathscr{U}}} && {\tilde{U}} \\
	\\
	& {\mathscr{U}} && U
	\arrow["\pi", from=1-4, to=3-4]
	\arrow[from=1-2, to=3-2]
	\arrow["\psi"', from=3-2, to=3-4]
	\arrow[from=1-2, to=1-4]
	\arrow["\sim", from=1-1, to=1-2]
	\arrow["{\pi'}"', from=1-1, to=3-2]
	\arrow["\lrcorner"{anchor=center, pos=0.125}, draw=none, from=1-2, to=3-4]
\end{tikzcd}\]

Univalence will imply that $\phi$ and $\psi$ are homotopy inverses of each other.
Precisely, the characterization \ref{univfib} of a fibration being univalent implies that the two bottom arrows in the following two diagrams ($\phi \circ \psi$ and $id_\mathscr{U}$) belong to the same contractible space (in particular, they are homotopic):

\[\begin{tikzcd}
	{\tilde{\mathscr{U}}} & {Z_{\mathscr{U}}} && {\tilde{U}} && {\tilde{\mathscr{U}}} && {\tilde{\mathscr{U}}} && {\tilde{\mathscr{U}}} \\
	\\
	& {\mathscr{U}} && U && {\mathscr{U}} && {\mathscr{U}} && {\mathscr{U}}
	\arrow["\pi", from=1-4, to=3-4]
	\arrow[from=1-2, to=3-2]
	\arrow["\psi"', from=3-2, to=3-4]
	\arrow[from=1-2, to=1-4]
	\arrow["\sim", from=1-1, to=1-2]
	\arrow["{\pi'}"', from=1-1, to=3-2]
	\arrow["\phi"', from=3-4, to=3-6]
	\arrow["{\pi'}", from=1-6, to=3-6]
	\arrow[from=1-4, to=1-6]
	\arrow["{ id_{\mathscr{U}}}"', from=3-8, to=3-10]
	\arrow["{\pi'}", from=1-10, to=3-10]
	\arrow["{\pi'}"', from=1-8, to=3-8]
	\arrow[from=1-8, to=1-10]
	\arrow["\lrcorner"{anchor=center, pos=0.125}, draw=none, from=1-8, to=3-10]
	\arrow["\lrcorner"{anchor=center, pos=0.125}, draw=none, from=1-4, to=3-6]
	\arrow["\lrcorner"{anchor=center, pos=0.125}, draw=none, from=1-2, to=3-4]
\end{tikzcd}\]

\noindent
Likewise, the following two diagrams correspond to two pullback squares in $Fun(\mathcal{C}^{op}, \mathcal{S})$ (hence in $\mathcal{C}$), so that $\psi \circ \phi$ and $id_U$ are also homotopic by univalence of the morphism $\tilde{\mathbf{u}} \to \mathbf{u}$.

\[\begin{tikzcd}
	{\tilde{U}} && {\tilde{\mathscr{U}}} & {Z_{\mathscr{U}}} & {\tilde{U}} && {\tilde{U}} && {\tilde{U}} \\
	\\
	U && {\mathscr{U}} && U && U && U
	\arrow["\pi", from=1-5, to=3-5]
	\arrow[from=1-4, to=3-3]
	\arrow["\psi"', from=3-3, to=3-5]
	\arrow[from=1-4, to=1-5]
	\arrow["\sim", from=1-3, to=1-4]
	\arrow["{\pi'}"', from=1-3, to=3-3]
	\arrow["{ id_U}"', from=3-7, to=3-9]
	\arrow["\pi", from=1-9, to=3-9]
	\arrow["\pi"', from=1-7, to=3-7]
	\arrow[from=1-7, to=1-9]
	\arrow["\lrcorner"{anchor=center, pos=0.125}, draw=none, from=1-7, to=3-9]
	\arrow["\lrcorner"{anchor=center, pos=0.125}, draw=none, from=1-4, to=3-5]
	\arrow["\phi"', from=3-1, to=3-3]
	\arrow[from=1-1, to=1-3]
	\arrow["\pi"', from=1-1, to=3-1]
	\arrow["\lrcorner"{anchor=center, pos=0.125}, draw=none, from=1-1, to=3-3]
\end{tikzcd}\]
 \end{proof}

Since $\tilde{\mathscr{U}}$ and $\mathscr{U}$ are equivalent to objects already in $\overline{\mathcal{C}}^*$, we deduce the following (in a sense, specializing Theorem \ref{t1}) :

\begin{theorem}
 \label{main}
 $\overline{\mathcal{C}}^{**}$ yields an interpretation of HoTT with one univalent universe.
\end{theorem}

\begin{proof}
 By construction, the fibration $\pi': \tilde{\mathscr{U}} \to \mathscr{U}$ provides a universe closed under $\Sigma$ and $\Pi$ type constructors and which classifies the "small" fibrations (where we mean the fibrations that are both $\kappa$-presentable and correspond to $\pi$-small morphisms in $\mathcal{C}$). 
 \end{proof}

Also, we note that the inclusion $\overline{\mathcal{C}}^{*} \hookrightarrow \overline{\mathcal{C}}^{**}$ is an equivalence.

\section{Adding disjoint sum, pushout and W-types}
\label{s4}

Substitution, as well as $\Sigma$-types, identity types, and $\Pi$-types, are arguably at the core of Martin-Löf type theory, and, as such, have required a specific treatment involving, for instance, strict pullbacks and dependent products. On the other hand, some other (higher) inductive type can be interpreted following a more general pattern, relying mainly on the existence of some structure in the quasicategory $\mathcal{C}$. For the definition of the types we are looking to add up, we follow \cite{ls2020} and \cite{lw2015local}.

In this section, we suggestively adapt our notations to the usual ones in type theory, and we think of objects of $\mathbf{SSet}^{\mathfrak{C}(\mathcal{C}^{op})}$ as (potential) contexts of the model of MLTT we are building, and fibrations between them as dependency (i.e, projections from an extended context to a base context). Precisely, we consider given a $\pi$-tribe $\overline{\mathcal{C}}^{**}$ as before, whose objects are the "current" contexts, and we look to add up new objects as to model new type constructors. Keep in mind that we want to do this while maintaining the condition that all the objects playing the role of contexts/types must be fibrant and essentially representable.

We start with disjoint sum ($+$-types).

\begin{proposition}
\label{sumtype}
 Suppose that $\mathcal{C}$ has finite coproducts.
 
 Consider two types $A$ and $B$ in context $\Gamma$ corresponding to two fibrations $A \to \Gamma$ and $B \to \Gamma$.
 
 It is possible to construct in $\mathbf{SSet}^{\mathfrak{C}(\mathcal{C}^{op})}$ an associated $+$-type $A+B$ whose underlying object in essentially representable.
 
 \end{proposition}
 
 \begin{proof}
 
 Let $a$, $b$ and $\gamma$ be objects of $\mathcal{C}$ such that $y(a)$, $y(b)$ and $y(\gamma)$ are equivalent to $A$, $B$ and $\Gamma$ (these exist since all the objects of $\overline{\mathcal{C}}^{**}$ are equivalent to objects of the form $\overline{y}(x)$ for $x$ an object of $\mathcal{C}$), and consider weak equivalences $y(a) \to A$, $y(b) \to B$ and $y(\gamma) \to \Gamma$ (this is possible since $y(a)$, $y(b)$ and $y(\gamma)$ are cofibrant and $A$, $B$ and $\Gamma$ are fibrant).
 Let $a \to \gamma$ and $b \to \gamma$ be maps such that $y(a) \to y(\gamma)$ and $y(b) \to y(\gamma)$ are equivalent to the fibrations $A \to \Gamma$ and $B \to \Gamma$.
 
 First, factor $y(a) \amalg y(b) \to y(a\amalg b)$ as a cofibration followed by a trivial fibration (write $Q$ for the middle object), then, form the Reedy (trivial cofibration, pointwise fibration)-factorization of the outer commuting square below,
 
\[\begin{tikzcd}[ampersand replacement=\&]
	{y(a) \amalg y(b)} \&\& P \&\& {A \amalg B} \\
	\\
	\\
	Q \&\& {A+B} \&\& \Gamma \\
	\& {y(a \amalg b)} \&\& {y(\gamma)}
	\arrow["\sim"{description}, hook, from=1-1, to=1-3]
	\arrow["\sim"{description}, curve={height=-12pt}, hook, from=1-1, to=1-5]
	\arrow[hook, from=1-1, to=4-1]
	\arrow["\sim"{description}, two heads, from=1-3, to=1-5]
	\arrow[hook, from=1-3, to=4-3]
	\arrow["s", curve={height=-12pt}, dashed, from=1-5, to=1-3]
	\arrow[from=1-5, to=4-5]
	\arrow["\sim"{description}, hook, from=4-1, to=4-3]
	\arrow["\sim"{description}, curve={height=6pt}, two heads, from=4-1, to=5-2]
	\arrow[two heads, from=4-3, to=4-5]
	\arrow[from=5-2, to=5-4]
	\arrow[curve={height=6pt}, from=5-4, to=4-5]
\end{tikzcd}\]
where $s$ is a section of the trivial fibration $P \to A \amalg B$.

 We will use both the universal properties of the coproduct in $\mathcal{C}$ and in $\mathbf{SSet}^{\mathfrak{C}(\mathcal{C}^{op})}$, and we will rely on the correspondence (unique up to homotopy) between morphisms $x \to z$ in $\mathcal{C}$ and maps $X \to Z$ in $\mathbf{SSet}^{\mathfrak{C}(\mathcal{C}^{op})}$, where $X$ (resp. $Z$) if a fibrant object equivalent to $y(x)$ (resp. $y(z)$).

We now have two injections $(\texttt{inl},\texttt{inr}): A \amalg B \rightarrow (A+B)_0 \rightarrow A+B$ that lie over $\Gamma$ by construction.

Suppose given a type $C$ (with an equivalence $y(c) \to C$) in context $\Gamma, A+B$ (so there is an edge $c \to a \amalg b$), and a morphism $d: A \amalg B \to C$ such that $p_C \circ d = (\texttt{inl},\texttt{inr})$.
These correspond to morphisms $a \to c$ and $b \to c$ in $\mathcal{C}$. There is an induced morphism (defined up to homotopy) $a \amalg b \to c$, which is moreover a section of $c \to a \amalg b$.
We deduce the existence of a morphism $\texttt{case}_0: A+B  \simeq \overline{y}(a \amalg b) \to \overline{y}(c) \rightarrow C$ (we first construct the morphism in homotopy category $\mathbf{Ho}(\mathbf{SSet}^{\mathfrak{C}(\mathcal{C}^{op})})$, then take a morphism realizing it: this is possible because both $A+B$ and $C$ are fibrant-cofibrant), that is a section of $p_C: \Gamma, A+B, C \to \Gamma, A+B$ up to homotopy, and such that $\texttt{case}_0 \circ (\texttt{inl},\texttt{inr})$ is homotopic to $d$.

Firstly, consider a homotopy $H_0: (A+B) \otimes \Delta_1 \to A+B$ from $p_C \circ \texttt{case}_0$ to $id_{A+B}$. The second component of the lift $K_0$ in the following diagram yields the expected section $\texttt{case}_0'$.

\[\begin{tikzcd}
	{A+B} && C \\
	\\
	{(A+B)\otimes \Delta_1} && {A+B}
	\arrow["\sim"{description}, hook, from=1-1, to=3-1]
	\arrow["{H_0}"', from=3-1, to=3-3]
	\arrow["{p_C}", two heads, from=1-3, to=3-3]
	\arrow["{\text{case}_0}", from=1-1, to=1-3]
	\arrow["{K_0}", dashed, from=3-1, to=1-3]
\end{tikzcd}\]

Secondly, consider a homotopy $$H: A \amalg B \to \mathbf{Path}(C)$$ from $\texttt{case}_0' \circ (\texttt{inl},\texttt{inr})$ to $d$. We take this homotopy in $\mathbf{SSet}^{\mathfrak{C}(\mathcal{C}^{op})}_{/A+B}$, which is possible because the composite $a \amalg b \to c \to a \amalg b$ being the identity morphism $a \amalg b \to a \amalg b$ really corresponds to a homotopy in $\mathcal{C}_{/a \amalg b}$.
We consider a solution $K$ to the following lifting problem: 

\[\begin{tikzcd}
	{A \amalg B} && {\mathbf{Path}(C)} \\
	\\
	{A+B} && C
	\arrow["\sim"{description}, two heads, from=1-3, to=3-3]
	\arrow[hook, from=1-1, to=3-1]
	\arrow["K", dashed, from=3-1, to=1-3]
	\arrow[from=3-1, to=3-3]
	\arrow["H", from=1-1, to=1-3]
\end{tikzcd}\]

\noindent
The second component of the homotopy $K$ is a morphism $\texttt{case}: A+B \to C$ such that $\texttt{case} \circ (\texttt{inl},\texttt{inr}) = d$.
\end{proof}

We record a similar result for pushout types:

\begin{proposition}
 Suppose that $\mathcal{C}$ has pushouts.

 Consider three types $A$, $B_1$ and $B_2$ in context $\Gamma$ corresponding to fibrations $A \to \Gamma$ and $B_i \to \Gamma$. Let $f_1 : A \to B_1$ and $f_2 : A \to B_2$ be morphisms above $\Gamma$.
 
 It is possible to construct, in $\mathbf{SSet}^{\mathfrak{C}(\mathcal{C}^{op})}$, an associated pushout type $D$ whose underlying object is essentially representable.
 \end{proposition}
 
 \begin{proof}
 The proof is almost identical to that of Proposition \ref{sumtype}, using the homotopy pushout $B_1 {\amalg}_{A \otimes \Delta_1} B_2$ instead of the coproduct in $\mathbf{SSet}^{\mathfrak{C}(\mathcal{C}^{op})}$. It is easy to see that we can construct a $\Delta_1$-typal pushout as in Definition 4.1 of \cite{ls2020}
 \end{proof}
 
The previous proof pattern seems very general. For instance, we can extend it to $W$-types, provided we work in a category of algebra.

 Let $\mathbf{C}$ be a simplicially enriched category, and consider a polynomial endofunctor $E : \mathbf{SSet}^{\mathbf{C}^{op}}_{/Y} \to \mathbf{SSet}^{\mathbf{C}^{op}}_{/Y}$
\[\begin{tikzcd}[ampersand replacement=\&]
	{\mathbf{SSet}^{\mathbf{C}^{op}}_{/Y}} \&\& {\mathbf{SSet}^{\mathbf{C}^{op}}_{/X}} \&\& {\mathbf{SSet}^{\mathbf{C}^{op}}_{/Z}} \&\& {\mathbf{SSet}^{\mathbf{C}^{op}}_{/Y}}
	\arrow["{H^*}", from=1-1, to=1-3]
	\arrow["{F_*}", from=1-3, to=1-5]
	\arrow["{G_!}", from=1-5, to=1-7]
\end{tikzcd}\]
where $H : X \to Y$, $F : X \to Z$ and $G : Z \to Y$ are morphisms in $\mathbf{SSet}^{\mathbf{C}^{op}}$.
Consider the algebraically-free monad $\mathbb{T}$ on $E$, and the fibrant replacement monad $\mathbb{R}_Y$ on $\mathbf{SSet}^{C^{op}}_{/Y}$ as in \cite{ls2020}. Consider as well their algebraic coproduct $\mathbb{T} + \mathbb{R}_Y$.
\begin{proposition}
\label{alg_ms}
 The category of algebra $\mathbf{Alg}_{\mathbb{T} + \mathbb{R}_Y}$ admits a model structure that is right transferred by the forgetful functor $\mathbf{Alg}_{\mathbb{T} + \mathbb{R}_Y} \to \mathbf{SSet}^{C^{op}}_{/Y}$.
\end{proposition}

\begin{proof}
 We will apply Lemma 2.3 of \cite{schwede2000algebras}, or rather a slightly more general version: we ask that the monad $T$ commutes with $\kappa$-sequential colimits for some cardinal $\kappa$. The proof given there still applies, provided we only consider a factorization when it is obtained by the small object argument, as a cell complex built as a $\lambda$-sequential colimit for $\lambda \geq \kappa$ followed by a (trivial) fibration.
 
  Consider a cardinal $\kappa$ such that $\mathbb{T} + \mathbb{R}_Y$ is $\kappa$-accessible. As discussed in \cite{nlab:transfinite_construction_of_free_algebras}, we know that $\mathbf{Alg}_{\mathbb{T} + \mathbb{R}_Y}$ is locally presentable, and moreover $\kappa$-sequential colimit are computed in $\mathbf{SSet}^{C^{op}}_{/Y}$ since $\mathbb{T} + \mathbb{R}_Y$ preserves them.
  
  As stated in \cite{schwede2000algebras}, this lemma applies when every $T$-algebra admits a path object, and every object is fibrant. However, the proof only relies on every object admitting a $T$-algebra structure to be fibrant, which is certainly the case for $T = \mathbb{T} + \mathbb{R}_Y$. 
  
  We still need to check the first condition. To do so, consider a $\mathbb{T} + \mathbb{R}_Y$-algebra over $A$ and a path object $PA$ for $A$ in $\mathbf{SSet}^{C^{op}}_{/Y}$. 
  We can provide $PA$ with a $\mathbb{R}_Y$-algebra structure, making the factorization a factorization in the category of $\mathbb{T} + \mathbb{R}_Y$-algebras. By definition, such a structure amounts to a coherent lifting operation of a trivial cofibration against $PY \to *$. We define such lifting operation as follows: given a lifting problem as below left, we first use the $\mathbb{R}_Y$-algebra structure on $A \times A$ to provide the lift $h$, and then use any (fixed) lifting operation against the fibration $PA \to A \times A$ to further get a lift $h'$, as below right, which is our chosen solution to the initial lifting problem.
  
\[\begin{tikzcd}[ampersand replacement=\&]
	I \&\& PA \&\& {A \times A} \&\& I \&\& PA \\
	\\
	J \&\& {*} \&\&\&\& J \&\& {A \times A}
	\arrow[from=1-1, to=1-3]
	\arrow["\sim"{description}, hook, from=1-1, to=3-1]
	\arrow[two heads, from=1-3, to=1-5]
	\arrow[two heads, from=1-3, to=3-3]
	\arrow[from=1-5, to=3-3]
	\arrow[from=1-7, to=1-9]
	\arrow["\sim"{description}, hook, from=1-7, to=3-7]
	\arrow[two heads, from=1-9, to=3-9]
	\arrow["h"{pos=0.3}, dashed, from=3-1, to=1-5]
	\arrow[from=3-1, to=3-3]
	\arrow["{h'}", dashed, from=3-7, to=1-9]
	\arrow["h"', from=3-7, to=3-9]
\end{tikzcd}\]
  It also follows that $PA \to A \times A$ is a morphism of $\mathbb{R}_Y$-algebras. Moreover, the fact that the diagonal $A \to A \times A$ is a morphism of $\mathbb{R}_Y$-algebras implies the same for the reflexivity map $A \to PA$.
  We now provide an $E$-algebra structure for $PA$. To do so, remark that the morphism $E(A) \to E(PA)$ is a monomorphism, since $E$ preserves monomorphisms, and it is also a weak equivalence by Ken Brown's lemma, since $E$ preserves trivial fibration. Therefore, it is a trivial cofibration, and the following lifting problem admits a solution:
\[\begin{tikzcd}[ampersand replacement=\&]
	{E(A)} \&\& A \&\& PA \\
	\\
	{E(PA)} \&\& {E(A \times A)} \&\& {A \times A}
	\arrow[from=1-1, to=1-3]
	\arrow["\sim"{description}, two heads, from=1-1, to=3-1]
	\arrow[from=1-3, to=1-5]
	\arrow[two heads, from=1-5, to=3-5]
	\arrow[dashed, from=3-1, to=1-5]
	\arrow[from=3-1, to=3-3]
	\arrow[from=3-3, to=3-5]
\end{tikzcd}\]
This solution provides the $E$-algebra structure, making the factorization a factorization in the category of $E$-algebras.
Overall, we have shown that $PA$ admits a $\mathbb{T} + \mathbb{R}_Y$-algebra structure and that the factorization $A \to PA \to A \times A$ is promoted to a path-object factorization in the category of $\mathbb{T} + \mathbb{R}_Y$-algebras. We can now apply the lemma to conclude.
\end{proof}

\begin{remark}
 Since $T$ is a homotopical functor (it is left Quillen and all objects are cofibrant), it is clear by definition of a category of algebras on monad that $\mathbf{Ho}(\mathbf{Alg}_T)$ is equivalent to $\mathbf{Ho}(\mathbf{Alg}_t)$, where $t : \mathcal{S}^{C^{op}} \to \mathcal{S}^{C^{op}}$ is the $\infty$-monad on the quasicategory of presheaves corresponding to $\mathbf{SSet}^{C^{op}}$.
\end{remark}

\begin{proposition}
 Consider two types $A$ and $B$, in context $\Gamma$ and $\Gamma, A$ respectively, corresponding to two fibrations $H: A \to \Gamma$ and $F: B \to A$. 
 These correspond (up to equivalence) to maps $h: a \to \gamma$ and $f: b \to a$ in $\mathcal{C}$. 
 
 Suppose that the polynomial endofunctor $p$ associated with $h$, $f$ and $g := h \circ f$ (as in the diagram below) admits an initial algebra.
\[\begin{tikzcd}[ampersand replacement=\&]
	{\mathcal{C}_{\gamma}} \&\& {\mathcal{C}_{b}} \&\& {\mathcal{C}_a} \&\& {\mathcal{C}_{\gamma}}
	\arrow["{h^*}", from=1-1, to=1-3]
	\arrow["{f_*}", from=1-3, to=1-5]
	\arrow["{g_!}", from=1-5, to=1-7]
\end{tikzcd}\]
\noindent
 Then we can construct in $\mathbf{SSet}^{\mathfrak{C}(\mathcal{C}^{op})}$ the corresponding $W$-type, and the underlying object is essentially representable. 
\end{proposition}

\begin{proof}

 Consider the initial algebra $px \to x$ in $\mathcal{C}$.
 Since $H^*$ corresponds to $h^*$ up to equivalence (in the sense that $H^* \overline{y}(x) \simeq \overline{y}(h^*x)$), and similarly for $F_*$, it is clear that $\overline{y}(x)$ can be made into a $P$-algebra $l: P\overline{y}(x) \to \overline{y}(x)$, where $P$ is the following polynomial endofunctor corresponding to $p$:

\[
\adjustbox{scale=0.9}{%
\begin{tikzcd}
	{\mathbf{SSet}^{\mathfrak{C}(\mathcal{C}^{op})}_{/\Gamma}} && {\mathbf{SSet}^{\mathfrak{C}(\mathcal{C}^{op})}_{/A}} && {\mathbf{SSet}^{\mathfrak{C}(\mathcal{C}^{op})}_{/B}} && {\mathbf{SSet}^{\mathfrak{C}(\mathcal{C}^{op})}_{/\Gamma}}
	\arrow["{H^*}", from=1-1, to=1-3]
	\arrow["{F_*}", from=1-3, to=1-5]
	\arrow["{G_!}", from=1-5, to=1-7]
\end{tikzcd}}\]

Proceeding as in \cite{ls2020}, write $\mathbb{T}_P$ for the algebraically-free monad on $P$, and $\mathbb{R}_\Gamma$ for the fibrant replacement monad (over $\Gamma$). Consider the initial $\mathbb{T}_P + \mathbb{R}_\Gamma$ algebra $X$. In particular, we have a $P$-algebra map $X \to \overline{y}(x)$. By \cref{alg_ms}, we know that $\mathbf{Alg}_{\mathbb{T}_P + \mathbb{R}_\Gamma}$ admits a right transferred model structure (from the forgetful functor).

Now form the (cofibration,trivial fibration)-factorization as in the diagram below in $\mathbf{Alg}_{\mathbb{T}_P + \mathbb{R}_\Gamma}$:

\[\begin{tikzcd}[ampersand replacement=\&]
	X \&\& W \&\& {\overline{y}(x)}
	\arrow[hook, from=1-1, to=1-3]
	\arrow["\sim"{description}, two heads, from=1-3, to=1-5]
\end{tikzcd}\]

We claim that $W$ is the expected $W$-type. The $\mathbf{Alg}_{\mathbb{T}_P + \mathbb{R}_\Gamma}$-algebra structure on $W$ yields the map $\mathbf{fold}: \Gamma, A, \Pi_B W[p_A] (= PW) \to \Gamma, W$ (over $\Gamma$). For the elimination rule, consider a type $C$ in context $\Gamma, W$ (and an equivalence $y(c) \to C$), and a square of the following form:

\[\begin{tikzcd}
	{\Gamma, A, \Pi_BW, \Pi_BC (=P\Sigma_WC)} && {\Gamma, A, \Pi_BW} \\
	\\
	{ \Gamma, W, C} && {\Gamma,W}
	\arrow["{\mathbf{fold}}", from=1-3, to=3-3]
	\arrow[from=3-1, to=3-3]
	\arrow[from=1-1, to=1-3]
	\arrow["D"', from=1-1, to=3-1]
\end{tikzcd}\]

There is a corresponding commutative square in $\mathcal{C}$:

\[\begin{tikzcd}
	pc && px \\
	\\
	c && x
	\arrow[from=1-3, to=3-3]
	\arrow[from=3-1, to=3-3]
	\arrow[from=1-1, to=1-3]
	\arrow["d"', from=1-1, to=3-1]
\end{tikzcd}\]

\noindent
By initiality of $px \to x$, we get a square (over $x$) as below,

\[\begin{tikzcd}
	px && pc \\
	\\
	x && c
	\arrow[from=1-1, to=3-1]
	\arrow["{\mathit{wrec}}"', from=3-1, to=3-3]
	\arrow[from=1-1, to=1-3]
	\arrow["d"', from=1-3, to=3-3]
\end{tikzcd}\]

\noindent
where $\mathit{wrec}$ is section of $c \to x$ in the quasicategory of algebras. This is also a section in the quasicategory of algebra of the $\infty$-monad $p'$ on $\mathcal{S}^{\mathcal{C}^{op}}$ defined from the representable maps associated to $h$, $f$, and $g$.

The corresponding map $\mathbf{wrec}_0$ is only a section of $p_C : C \to W$ up to homotopy (that is, it is a section in homotopy category $\mathbf{Alg}_{\mathbb{T}_P + \mathbb{R}_\Gamma}$), but we can replace it with a homotopic map that is a real section. To do that, working with the model structure over $\mathbf{SSet}^{\mathfrak{C}(\mathcal{C}^{op})}$, we consider a homotopy $H_0$ from $p_C \circ \mathbf{wrec}_0$ to $id_W$ (noting that $p_C$ is fibration in $\mathbf{Alg}_{\mathbb{T}_P + \mathbb{R}_\Gamma}$ because the notion of fibration there is transferred from the forgetful functor $\mathbf{Alg}_{\mathbb{T}_P + \mathbb{R}_\Gamma} \to \mathbf{SSet}^{\mathfrak{C}(\mathcal{C}^{op})}$). The second component of the lift in the following diagram gives the section $\mathbf{wrec}$.

\[\begin{tikzcd}[ampersand replacement=\&]
	W \&\& C \\
	\\
	{\mathbf{Cyl}(W)} \&\& W
	\arrow["{\mathbf{wrec}_0}", from=1-1, to=1-3]
	\arrow["\sim"{description}, hook, from=1-1, to=3-1]
	\arrow["{p_C}", two heads, from=1-3, to=3-3]
	\arrow["{K_0}", dashed, from=3-1, to=1-3]
	\arrow["{H_0}"', from=3-1, to=3-3]
\end{tikzcd}\]
\end{proof}

\begin{remark}
 Given a concrete description of the associated syntax, we believe that any cell monad with parameters (in the sense of \cite{ls2020}), inside of $\mathcal{C}$, admitting an initial algebra, yields a corresponding higher inductive type in $\mathbf{SSet}^{\mathfrak{C}(\mathcal{C}^{op})}$ using the previous proof pattern. However, we do not further investigate this question here.
\end{remark}

\section{Elementary higher topoi}

The following definition for the notion of elementary higher topos was originally proposed by Shulman (\cite{ncafe}), and further studied by Rasekh (see \cite{rasekh2018eth}):

\begin{definition}
 \label{eht}
 An elementary higher topos is a quasicategory $\mathcal{E}$ such that 
 
 \begin{enumerate}
  \item $\mathcal{E}$ is finitely complete and cocomplete.
  \item $\mathcal{E}$ is locally cartesian closed.
  \item $\mathcal{E}$ admits a subobject classifier.
  \item For every morphism $f: x \to y$ in $\mathcal{E}$, there exists an object classifier $p: V \to U$ classifying $f$, such that the class of morphisms it classifies is closed under composition, fiberwise finite limits and colimits, as well as dependent products.
 \end{enumerate}

\end{definition}

At this point, we have studied how the properties 1, 2, and 4 in Definition \ref{eht} can be used to provide features of the type theory we construct from the quasicategory $\mathcal{C}$ (when taking $\mathcal{C} := \mathcal{E}$). Namely, we have that :
\begin{itemize}
 \item Pullbacks provide substitutions (Section \ref{s2}), the terminal object plays the role of the empty context, and pushout give the so-called "pushout types" (Section \ref{s4}).
 \item Local cartesian closure yields $\Pi$-types (Section \ref{s2}).
 \item Object classifiers, or rather, equivalently, univalent maps (see Theorem 3.32 of \cite{rasekh2018eth}), provide univalent universes (Section \ref{s3}). 
\end{itemize}

We now consider property 3, that is, the existence of a subobject classifier.
The link between such an object and the propositional resizing axiom is known (this is notably discussed in Section 8 of \cite{rasekh2018pr}).

We will use the following definition:

\begin{definition}
\label{propres}
 A model of type theory (e.g. a tribe) satisfies propositional resizing if there exists a fibration $s : \tilde{V} \to V$, classifying the fibrations that are $(-1)$-truncated (the latter notion being defined representably from the corresponding notion on spaces, or equivalently from contractibility of its fibers, in the usual homotopy type-theoretic sense).
\end{definition}

The following proposition follows from Section \ref{s3} of the present chapter, in the special case where the class $\mathbf{S}$ of morphisms classified by the universe is the class of fibrations that are homotopy monomorphisms ($(-1)$-truncated maps).

\begin{proposition}
 Suppose $\mathcal{C}$ has a subobject classifier. Then there is an essentially representable object of $\mathbf{SSet}^{\mathfrak{C}(\mathcal{C}^{op})}$ that witnesses propositional resizing in the sense of Definition \ref{propres}.
\end{proposition}

Observing that a family of univalent universes inside $\mathcal{C}$ yields a family of univalent universes in the $\pi$-tribe $\overline{\mathcal{C}}^{**}$ (as in Section 3), Theorem \ref{main} has now the following corollary:

\begin{theorem}
 Every elementary higher topos provides a model for homotopy type theory.
\end{theorem}

\begin{remark}
 Applying repeatedly the argument of Section \ref{s3} still requires strong set-theoretic assumptions (at least countably many inaccessible cardinals). This is the usual requirement, and it cannot be avoided (in the sense that the model of homotopy type theory constructed this way yields back a model of ZFC + "there exist infinitely many inaccessible cardinals").
\end{remark}

\begin{example}
 In his paper \cite{anel2021}, Anel observes that the quasicategory $\mathcal{S}^{< \infty}_{\text{coh}}$ of truncated coherent spaces has properties that ought to make it an example of elementary higher topos, although it does not admit all finite colimits, and its universes are not closed under type constructors (dependent sum and product). In particular, it does not satisfy the axioms of Definition \ref{eht}, and we cannot apply directly the previous theorem (nor can we apply Theorem 11.2 from \cite{shulman2019}). Our work from Section \ref{s3}-\ref{s4} shows that, nonetheless, the quasicategory $\mathcal{S}^{< \infty}_{\text{coh}}$ provides an interpretation for a homotopy type theory that lacks some higher inductive types, and whose universes are not closed under type constructors.
\end{example}

\newpage

\section*{Appendix - Some details on the interpretation of MLTT}
\addcontentsline{toc}{section}{Appendix - Some details on the interpretation of MLTT}

In this appendix, we quickly recall the usual interpretation of the core of MLTT in a (full subcategory of a) suitable model category, or, more generally, in simplicial $\pi$-tribe $\mathcal{T}$.

\begin{definition}[Interpretation]

We interpret the core of MLTT as follows:
 
 \begin{itemize}
  \item Each context $\Gamma$ is interpreted by a an object of $\mathcal{T}$ noted $\llbracket \Gamma \rrbracket$.
  \item The empty context is interpreted by the terminal object $id_\mathcal{C}: \mathcal{C} \to \mathcal{C}$
  \item Each substitution $\Gamma \vdash s: \Delta$ is interpreted by a morphism $\llbracket \Gamma \rrbracket \to \llbracket \Delta \rrbracket$.
  \item Each dependent type $\Gamma \vdash A$ is interpreted by a fibration $p_A: \llbracket \Gamma, A \rrbracket \to \llbracket \Gamma \rrbracket$.
  \item Each term $\Gamma \vdash a: A$ is interpreted by a section of $p_A$.
  
 \end{itemize}

\end{definition}

\begin{proposition}[Substitution]
 Substitution can be modeled by pullback: if $\Gamma \vdash s: \Delta$ is a substitution and $\Delta \vdash A$ is a type in context $\Delta$, there is a substituted type $\Gamma \vdash s^* A$
\end{proposition}

\begin{proof}
 Since fibrations are stable under pullback, pulling back $s$ along $p_A: \llbracket \Delta, A \rrbracket \to \llbracket \Delta \rrbracket$ yields a fibration $p'_A: \llbracket \Gamma, s^*A \rrbracket \to \llbracket \Gamma \rrbracket$. Given a term $\Delta \vdash a: A$, the universal property of the pullback provides a term $a'$ of type $s^*A$.
\end{proof}

\begin{proposition}[Dependent sum]
 Given types $\Gamma \vdash A$ (associated to a projection $p_A$) and $\Gamma, A \vdash B$ (associated to a projection $p_B$), there is a sum-type $\Gamma \vdash \Sigma_A B$ provided by the composite $p_{\Sigma_A B} = p_A \circ p_B$.
\end{proposition}

\begin{proof}
 The composite of two fibrations indeed yields a fibration. By construction $\Gamma, A, B$ is the same as the context $\Gamma, \Sigma_A B$. The identity morphism associated to the object $\llbracket \Gamma, A, B \rrbracket$ yields the substitution $\Gamma, A, B \vdash pair_{A,B}: \Sigma_A B$.
 
 Given a term $\Gamma \vdash a: A$ (given by a section $s_A$ of $p_A$) and a term $\Gamma \vdash B[a]$ (as section $s_B$ of $p_B$); $s_B \circ s_A$ gives a section of $p_{\Sigma_A B}$ which is the term $(a,b)$.
 For elimination, suppose given a type $\Gamma, \Sigma_A B \vdash C$ (corresponding to the projection $p_C$) and a term $\Gamma, x: A, y: B \vdash d(x,y): C$ (as a section of the pullback of $p_C$ along $pair_{A,B}$). This gives immediately a term $\Gamma, z: \Sigma_A B \vdash d(z): C$ which is the exact same section (the pullback is trivial).
 
 The compatibility with substitution stems from the Beck-Chevalley condition satisfied by the adjunctions of the form $f_! \dashv f^*$. 
\end{proof}

\begin{proposition}[Dependent product]

 Given types $\Gamma \vdash A$ (associated to a projection $p_A$) and $\Gamma, A \vdash B$ (associated to a projection $p_B$), there is a product-type $\Gamma \vdash \Pi_A B$ provided by the dependent product $\Pi_{p_A} p_B$.
 
\end{proposition}

\begin{proof}
 The dependent product $p_{\Pi_A B} = \Pi_{p_A} p_B$ is still a fibration. By transposition, a term $\Gamma, a: A \vdash b: B$ given as section $s_B$ of $p_B$ is equivalent tot the data of a section $s_\Pi$ of $p_{\Pi_A B}$.
 
 Explicitly, given a term $\Gamma \vdash f: \Pi_A B$ given as a section $s_\Pi$, and a term $\Gamma \vdash a: A$ (a section $s_A$). The applied term $\Gamma, a: A \vdash f(a): B$ is interpreted by the composite of $s_A$ and of the transpose $s_\Pi'$ of $s_\Pi$. In particular $p_B \circ s_\Pi' = s_A$.
 
 Compatibility with substitution is as before.
 \end{proof}

\begin{proposition}[Identity types]
  Given types $\Gamma \vdash A$ (associated to a projection $p_A$), there is a identity-type $\Gamma, A, p_A^*A \vdash Id_A$ provided by the path object $\mathbf{Path}(p_A)$ (obtained by factoring the dependent diagonal $A \to \Delta_{p_A} A$ as $(i_0,i_1) \circ \mathbf{refl}_A$).
\end{proposition}

\begin{proof}
 The so defined identity type comes by definition with a fibration $p_{Id_A} = \mathbf{Path}(p_A) \to p_A^*A$ a section $\mathbf{refl}_A$ of $p_{p_A^*A} \circ p_{Id_A}$ (the first part of the factorization).
 
 Given a type $\Gamma, A, p_A^*A, Id_A \vdash C$ with terms given as a section $d: \Gamma, A \to \Gamma, A, p_A^*A, Id_A, C$ such that $p_C \circ d = \mathbf{refl}_A$, there is a section $J_{C,d}:  \Gamma, A, p_A^*A, Id_A \to \Gamma, A, p_A^*A, Id_A, C$ given as a solution of the following lifting problem (hence $J_{C,d} \circ \mathbf{refl}_A = d$).
 
\[\begin{tikzcd}
	{\Gamma, A} && {\Gamma, A, p_A^*A, Id_A, C} \\
	\\
	{\Gamma, A, p_A^*A, Id_A} && {\Gamma, A, p_A^*A, Id_A}
	\arrow["{p_C}", from=1-3, to=3-3]
	\arrow["{\mathbf{refl}_A}"', from=1-1, to=3-1]
	\arrow["id"', from=3-1, to=3-3]
	\arrow["d", from=1-1, to=1-3]
	\arrow["{J_{C,d}}"{description}, dashed, from=3-1, to=1-3]
\end{tikzcd}\]
\end{proof}

\begingroup
\setlength{\emergencystretch}{.5em}
\RaggedRight
\printbibliography
\endgroup

\end{document}